\documentclass[smallextended,numbook]{svjour3}
\smartqed
\usepackage[algoruled,algosection,lined,longend]{algorithm2e}
\usepackage{amsmath}
\usepackage{amssymb}
\usepackage{hyperref}
\usepackage{graphicx}
\usepackage{mathptmx}
\usepackage{mathtools}
\usepackage{paralist}
\usepackage{pifont}
\journalname{Numerical Algorithms}
\newcommand{\diag}{\ensuremath\mathop{\mathrm{diag}}}
\newcommand{\rank}{\ensuremath\mathop{\mathrm{rank}}}
\newcommand{\sech}{\ensuremath\mathop{\mathrm{sech}}}
\newcommand{\sign}{\ensuremath\mathop{\mathrm{sign}}}
\newcommand{\off}{\ensuremath\mathop{\mathrm{off}_F^2}}
\newcommand{\fma}{\ensuremath\mathop{\mathtt{fma}}}
\newcommand{\RE}{\ensuremath\mathop{\mathrm{Re}}}
\newcommand{\IM}{\ensuremath\mathop{\mathrm{Im}}}
\begin{document}
\title{A Kogbetliantz-type algorithm for the hyperbolic SVD}
\author{Vedran Novakovi\'{c}\and Sanja Singer}
\institute{%
  Vedran Novakovi\'{c} (corresponding author)\at
  completed a major part of this research as a collaborator on the MFBDA project, 10000 Zagreb, Croatia
  (\,\includegraphics[height=2ex,bb=40 52 296 308]{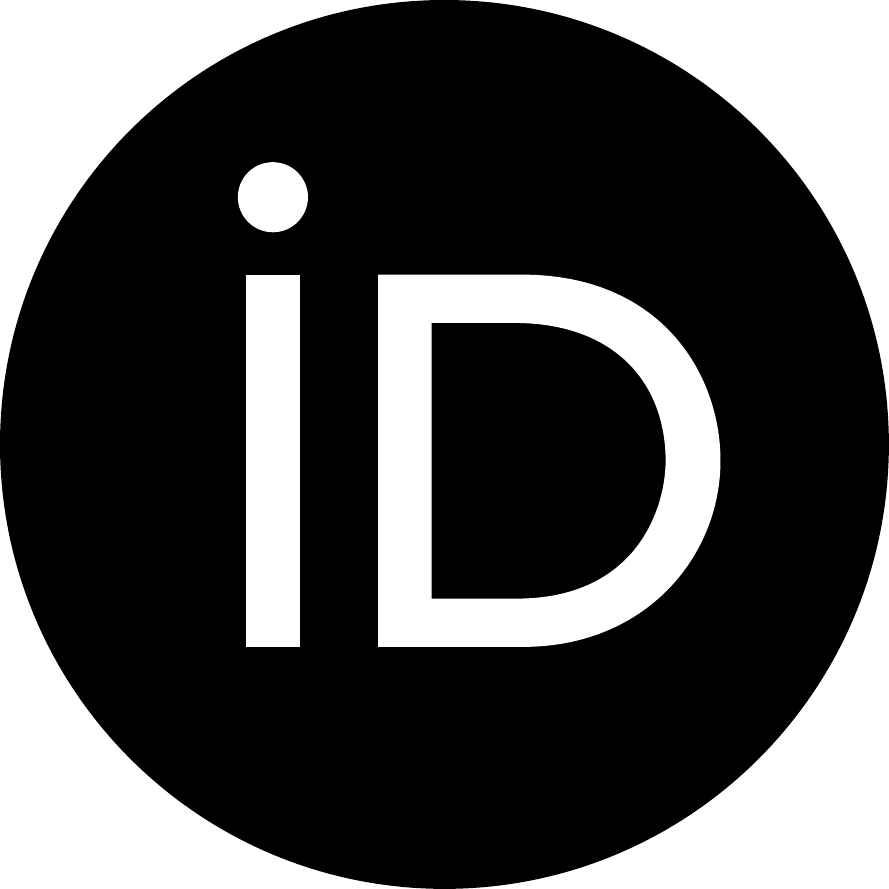}\,\href{https://orcid.org/0000-0003-2964-9674}{\texttt{0000-0003-2964-9674}}\,)\\
  \email{venovako@venovako.eu}%
  \and
  Sanja Singer$\null^{\dagger}$\at
  University of Zagreb, Faculty of Mechanical Engineering and Naval Architecture,
  Ivana Lu\v{c}i\'{c}a 5, 10000 Zagreb, Croatia
  (\,\includegraphics[height=2ex,bb=40 52 296 308]{orcid.pdf}\,\href{https://orcid.org/0000-0002-4358-1840}{\texttt{0000-0002-4358-1840}}\,)
}
\dedication{This work is dedicated to the memory of Sanja Singer.}
\date{Received: date / Accepted: date}
\maketitle
\begin{abstract}
  In this paper a two-sided, parallel Kogbetliantz-type algorithm for
  the hyperbolic singular value decomposition (HSVD) of real and
  complex square matrices is developed, with a single assumption that
  the input matrix, of order $n$, admits such a decomposition into the
  product of a unitary, a non-negative diagonal, and a $J$-unitary
  matrix, where $J$ is a given diagonal matrix of positive and
  negative signs.  When $J=\pm I$, the proposed algorithm computes the
  ordinary SVD\@.

  \looseness=-1
  The paper's most important contribution---a derivation of formulas
  for the HSVD of $2\times 2$ matrices---is presented first, followed
  by the details of their implementation in floating-point arithmetic.
  Next, the effects of the hyperbolic transformations on the columns
  of the iteration matrix are discussed.  These effects then guide a
  redesign of the dynamic pivot ordering, being already a
  well-established pivot strategy for the ordinary Kogbetliantz
  algorithm, for the general, $n\times n$ HSVD\@.  A heuristic but
  sound convergence criterion is then proposed, which contributes to
  high accuracy demonstrated in the numerical testing results.  Such a
  $J$-Kogbetliantz algorithm as presented here is intrinsically slow,
  but is nevertheless usable for matrices of small orders.
\end{abstract}
\section{Introduction}\label{s:1}
The Kogbetliantz algorithm~\cite{Kogbetliantz-55} is the oldest
effective method discovered that computes the singular value
decomposition (SVD) of a square matrix $G$ as $G=U\Sigma V^{\ast}$,
where $U$ and $V$ are unitary matrices, while $\Sigma$ is a diagonal
matrix with the non-negative diagonal elements, called the singular
values, that are usually ordered non-increasingly, i.e.,
$\Sigma=\diag(\sigma_1,\ldots,\sigma_n)$ and
$\sigma_k\ge\sigma_{\ell}\ge 0$, for $1\le k<\ell\le n$, with
$n$ being the matrix order of $G$, $U$, $\Sigma$, and $V$.

In this paper a Kogbetliantz-type algorithm for the hyperbolic
singular value decomposition~\cite{Onn-Steinhardt-Bojanczyk-91} (the
HSVD in short) is developed and called the $J$-Kogbetliantz algorithm
in the following.  Since the Kogbetliantz-type algorithms operate only
on square matrices, Definition~\ref{d:1.1} of the HSVD is restricted
to such a case.

\begin{definition}\label{d:1.1}
  Let $J$ and $G$ be two square matrices of order $n$, such that $J$
  is a diagonal matrix of signs, $J=\diag(\pm 1)$, and $G$ is a matrix
  over $\mathbb{F}\in\{\mathbb{R},\mathbb{C}\}$ with
  $\rank(G J G^{\ast})=\rank(G)$.  A decomposition of $G$,
  \begin{displaymath}
    G=U\Sigma V^{-1},
  \end{displaymath}
  such that $U$ is a unitary matrix over $\mathbb{F}$, $V$ is a
  $J$-unitary matrix over $\mathbb{F}$ with respect to $J$ (i.e.,
  $V^{\ast} J V = J$, so $V$ is hypernormal in the terminology
  of~\cite{Bojanczyk-Onn-Steinhardt-93}), and $\Sigma$ is a real
  diagonal matrix with a non-negative diagonal, is called the
  \emph{hyperbolic singular value decomposition} of $G$.  The diagonal
  elements of $\Sigma$ are called the
  \emph{hyperbolic singular values} of $G$, which are assumed to be
  ordered non-increasingly (non-decreasingly) in any, not necessarily
  contiguous, range of diagonal elements of $\Sigma$ for which the
  corresponding range of diagonal elements of $J$ contains only
  positive (negative) signs.
\end{definition}
When $J=\pm I_n$ the HSVD becomes the ordinary SVD, with $V$ unitary,
and the $J$-Kogbetliantz algorithm reduces to the ordinary
Kogbetliantz algorithm for the SVD.

In Definition~\ref{d:1.1} the assumption that
$\rank(G J G^{\ast})=\rank(G)$ ensures~\cite{Zha-96} that the HSVD of
$G$ exists, with a diagonal $\Sigma$.  If the assumption does not
hold, or if $G$ is rectangular, for the $J$-Kogbetliantz algorithm $G$
should be preprocessed by, e.g., the $J$-URV
factorization~\cite{SingerSanja-06,Singer-DiNapoli-Novakovic-Caklovic-20}
to a square matrix $G_0$ of order $n_0\le n$, such that
$\rank(G_0^{} J_0^{} G_0^{\ast})=\rank(G_0^{})$, i.e.,
$G=U_0^{\ast} \widetilde{G}_0^{} V_0^{}$, where $U_0$ is unitary, $V_0$
is $J$-unitary, $J_0=\diag(\pm 1)$ is of order $n_0$, and
$G_0=\widetilde{G}_0(1:n_0,1:n_0)$.  Otherwise, let $n_0=n$, $G_0=G$,
$V_0=U_0=I_n$, and $J_0=J$.

Besides an obvious application as the main part of a method for
solving a Hermitian indefinite
eigenproblem~\cite{Singer-et-al-12a,Singer-et-al-12b}, the HSVD has
also been used in various signal processing
applications~\cite{Onn-Steinhardt-Bojanczyk-91}, and recently in a
modified approach to the Kalman filtering, especially in the
ill-conditioned
case~\cite{Kulikov-Kulikova-20,Kulikova-19a,Kulikova-19b}.  The HSVD
can be efficiently and accurately computed by the \emph{one-sided}
blocked Jacobi-type
algorithms~\cite{Hari-SingerSanja-SingerSasa-10,Hari-SingerSanja-SingerSasa-14},
even on the massively parallel architectures~\cite{Novakovic-15}.
A \emph{two-sided}, Kogbetliantz-type method proposed here is not
intended to outperform the existing one-sided HSVD algorithms.
Instead, it is meant to showcase the tools (e.g., a careful dynamic
parallel ordering and a sound convergence criterion) designed to cope
with the perils of non-unitary transformations, that empirically
happen to be a greater risk for numerical stability of the two-sided
methods---here, for the hyperbolic, but possibly also applicable to
the more general, orthosymmetric~\cite{Mackey-Mackey-Tisseur-05}
SVD---than for that of the one-sided algorithms.

The ordinary Kogbetliantz algorithm usually transforms an upper
triangular matrix $R$ (see, e.g.,~\cite{Hari-Veselic-87}), resulting
from the preprocessing of $G$ by the QR factorization, in order to
simplify the transformations and lower their number.  If a particular
cyclic pivot ordering is applied, at the end of the first cycle a
lower triangular matrix is obtained, while after the subsequent cycle
the iteration matrix becomes upper triangular again.  Unfortunately, a
simple generalization of the Kogbetliantz algorithm to a hyperbolic
one (by introducing the hyperbolic transformations but leaving the
pivot strategy intact) can fail, since even a single hyperbolic
transformation from the right, with a sufficiently high condition
number, can cause an excessive growth of the off-diagonal elements of
the iteration matrix and ruin the convergence of the algorithm.
A different, dynamic (data-dependent) pivot strategy is therefore
needed to keep the sometimes unavoidable growth of the off-diagonal
norm in check.  As a consequence, the iteration matrix cannot be
assumed to have or retain any particular structure.

The $J$-Kogbetliantz algorithm computes the HSVD of $G_0$ in a
sequence of transformations, infinite in general but cut off to a
finite leading part when computing in machine precision and a
convergence criterion is met, as
\begin{displaymath}
  U^{\ast}G_0 V\approx\Sigma,\quad
  U^{\ast}=U_N^{\ast}U_{N-1}^{\ast}\cdots U_1^{\ast},\quad
  V=V_1^{}V_2^{}\cdots V_N^{},
\end{displaymath}
where, for each $k$, $1\le k\le N$, $U_k^{\ast}$ and $V_k^{}$ are the
embeddings of the $2\times 2$ transformations $\widehat{U}_k^{\ast}$
and $\widehat{V}_k^{}$, respectively, into $I_{n_0}^{}$, that are
applied to the iteration matrix $G_{k-1}^{}$, forming a sequence of
the iteration matrices
\begin{displaymath}
  G_1^{},G_2^{},\ldots,G_N^{}\approx\Sigma,
\end{displaymath}
where $G_k^{}=U_k^{\ast}G_{k-1}^{}V_k^{}$.  If $G_0$ is already in the
form required of $\Sigma$, then $V=U=I_{n_0}^{}$, $\Sigma=G_0^{}$, and
no transformations take place.  Else, $U_k^{\ast}$ is unitary
(orthogonal in the real case), while $V_k^{}$ is $J_0^{}$-unitary
($J_0^{}$-orthogonal in the real case).  Each $\widehat{U}_k^{\ast}$
is unitary (orthogonal), and $\widehat{V}_k^{}$ is
$\widehat{J}_k^{}$-unitary ($\widehat{J}_k^{}$-orthogonal), where
$\widehat{J}_k^{}=\diag(J_0^{}(p_k,p_k),J_0^{}(q_k,q_k))$ for the
$k$th pivot index pair $(p_k,q_k)$, $1\le p_k<q_k\le n_0$.  The
decomposition is finalized by forming
$U=(U^{\ast})^{\ast}=U_1^{}U_2^{}\cdots U_N^{}$, while by multiplying
$V^{\ast}J_0^{}V=J_0^{}$ from the left by $J_0^{}$, and noting that
$J_0^2=I_{n_0}^{}$, it follows that $V^{-1}=J_0^{}V^{\ast}J_0^{}$.

In the $k$th step the elements $(1,1)$, $(2,1)$, $(1,2)$, and $(2,2)$
of $\widehat{U}_k^{\ast}$ (and $\widehat{V}_k^{}$) are embedded into
$U_k^{\ast}$ (and $V_k^{}$) at the pivot positions $(p_k,p_k)$,
$(q_k,p_k)$, $(p_k,q_k)$, and $(q_k,q_k)$, respectively, where the
pivot indices are chosen such that for the pivot matrix
$\widehat{G}_{k-1}^{}$,
\begin{equation}
  \widehat{G}_{k-1}^{}=
  \begin{bmatrix}
    G_{k-1}^{}(p_k,p_k) & G_{k-1}^{}(p_k,q_k)\\
    G_{k-1}^{}(q_k,p_k) & G_{k-1}^{}(q_k,q_k)
  \end{bmatrix},
  \label{e:1.1}
\end{equation}
holds at least one of
\begin{displaymath}
  \widehat{G}_{k-1}^{}(2,1)\ne 0,\quad
  \widehat{G}_{k-1}^{}(1,2)\ne 0,\quad
  \widehat{G}_{k-1}^{}(1,1)\notin\mathbb{R}_0^+,\quad
  \widehat{G}_{k-1}^{}(2,2)\notin\mathbb{R}_0^+.
\end{displaymath}
When $\widehat{J}_k^{}(1,1)=\widehat{J}_k^{}(2,2)$ (such a case is
called \emph{trigonometric} in the following, while the other case is
called \emph{hyperbolic}), $\widehat{G}_{k-1}^{}$ is also a
transformation candidate if it is diagonal, with the non-negative
diagonal elements such that
$\widehat{G}_{k-1}^{}(1,1)<\widehat{G}_{k-1}^{}(2,2)$ if
$\widehat{J}_k^{}=I_2^{}$, or
$\widehat{G}_{k-1}^{}(1,1)>\widehat{G}_{k-1}^{}(2,2)$ if
$\widehat{J}_k^{}=-I_2{}$.  This is summarized in
Definition~\ref{d:1.2}.

\begin{definition}\label{d:1.2}
  In a step $k\ge 1$, let $p_k$ and $q_k$, $1\le p_k<q_k\le n_0$, be
  such indices for which $\widehat{G}_{k-1}$ from~\eqref{e:1.1} does
  \emph{not} have a form required of a matrix of hyperbolic singular
  values, as specified in Definition~\ref{d:1.1}.  Then,
  $\widehat{G}_{k-1}$ is called a transformation candidate or a pivot
  (sub)matrix.  If such indices do not exist, there are no
  transformation candidates in the $k$th step.  If more than one index
  pair satisfies the defining condition, one pivot index pair is
  selected according to the chosen pivot strategy.
\end{definition}

If the step index is omitted, then for the pivot index pair $(p,q)$
the transformation matrices ($\widehat{U}^{\ast}$ and $\widehat{V}$),
the pivot matrix ($\widehat{G}$), and the corresponding matrices of
signs ($\widehat{J}$) and hyperbolic singular values
($\widehat{\Sigma}$) can be written as
\begin{displaymath}
  \renewcommand{\arraycolsep}{2pt}
  \widehat{U}^{\ast}=
  \begin{bmatrix}
    \bar{u}_{pp} & \bar{u}_{qp}\\
    \bar{u}_{pq} & \bar{u}_{qq}
  \end{bmatrix}\!,\;
  \widehat{V}=
  \begin{bmatrix}
    v_{pp} & v_{pq}\\
    v_{qp} & v_{qq}
  \end{bmatrix}\!,\;
  \widehat{G}=
  \begin{bmatrix}
    g_{pp} & g_{pq}\\
    g_{qp} & g_{qq}
  \end{bmatrix}\!,\;
  \widehat{J}=
  \begin{bmatrix}
    j_{pp} & 0\\
    0 & j_{qq}
  \end{bmatrix}\!;\;
  \widehat{\Sigma}=
  \begin{bmatrix}
    \sigma_{pp} & 0\\
    0 & \sigma_{qq}
  \end{bmatrix}\!.
\end{displaymath}
With $\widehat{G}$ given, $\widehat{U}^{\ast}$ and $\widehat{V}$ are
sought for, such that
$\widehat{U}^{\ast}\widehat{G}\widehat{V}=\widehat{\Sigma}$, with
$\sigma_{pp}$ and $\sigma_{qq}$ being the hyperbolic singular values
of $\widehat{G}$, and
$\widehat{V}^{\ast}\widehat{J}\widehat{V}=\widehat{J}$.

If there are no transformation candidates in the $k$th step, or if the
convergence criterion is satisfied, the algorithm stops successfully,
with (an approximation of) the HSVD of $G$ found.  Otherwise,
$\widehat{U}_k^{\ast}$ and $\widehat{V}_k^{}$ are computed for one,
suitably chosen transformation candidate of $G_{k-1}^{}$, and applied
to the pivot rows $p_k$ and $q_k$ from the left
($\widehat{U}_k^{\ast}$) and the pivot columns $p_k$ and $q_k$ from
the right ($\widehat{V}_k^{}$), in that order for determinacy, to form
$G_k^{}$.  The process is then repeated in the step $k+1$.

Note that several, but at most $\lfloor n_0/2\rfloor$ successive steps
can be grouped to form a multi-step $\mathbf{k}$, where
$0\le|\mathbf{k}|\le\lfloor n_0/2\rfloor$ is the chosen length of
$\mathbf{k}$, if and only if $\{p_k,q_k\}\cap\{p_l,q_l\}=\emptyset$
for all $k\ne l$ such that $\{k,l\}\subseteq\mathbf{k}$, i.e., all
basic steps within a multi-step can be performed in \emph{parallel}
(with $t\ge 1$ parallel tasks at hand).  The number of basic steps may
vary from one multi-step to another if, e.g., not enough
transformation candidates are available, and may reach zero when no
further transformations are possible.  With $n_0=6$, e.g., the
multi-step $\mathbf{1}$ might be $\{1,2,3\}$, $\mathbf{2}$ might be
$\{4,5\}$, $\mathbf{3}$ might be $\{6,7,8\}$, continuing until some
$\mathbf{N}$ at which the algorithm halts.

The parallel application of the $2\times 2$ transformations has to
take into account that all transformations from one side (e.g., the
row transformations from the left) have to precede any transformation
from the other side (e.g., the column transformations from the right).
However, all transformations from the same side can proceed
concurrently, and the choice of the first side to be transformed is
arbitrary.

To fully describe the $J$-Kogbetliantz algorithm, it therefore
suffices to specify:
\begin{compactenum}
  \item a method for computing the HSVD of the pivot matrices of order
    two,
  \item the details of performing the row and the column
    transformations,
  \item a pivot strategy that selects the transformation candidate(s)
    in a (multi-)step, and
  \item a convergence criterion that halts the execution.
\end{compactenum}
The above list guides the organization of this paper as follows.  The
first item is covered in section~\ref{s:2}, the second one in
section~\ref{s:3}, the third one in section~\ref{s:4}, and the last
one in section~\ref{s:5}, containing an overview of the algorithm.
The numerical testing results are summarized in section~\ref{s:6}, and
the paper is concluded with some remarks on the future work in
section~\ref{s:7}.  Appendix~\ref{s:A} contains several technical
proofs.
\section{Computing the HSVD of matrices of order two}\label{s:2}
In the ordinary Kogbetliantz algorithm it is preferred that the pivot
matrices throughout the process remain (upper or lower)
triangular~\cite{Charlier-Vanbegin-VanDooren-87,Hari-Veselic-87} under
a cyclic pivot strategy: a serial (e.g., the row or the column cyclic,
with $G_0$ triangular) or a parallel one (e.g., the modulus strategy,
with $G_0$ preprocessed into the butterfly
form~\cite{Hari-Zadelj-Martic-07})---a fact relied upon for a simple
and accurate computation of the transformation
parameters~\cite{Hari-Matejas-09,Matejas-Hari-10,Matejas-Hari-15}.
However, the $J$-Kogbetliantz algorithm employs a strategy that has no
periodic (cyclic) pattern, so no particular form of the pivot matrices
can be guaranteed and no special form of $G_0$ is presumed.  Depending
on $\widehat{J}$ and its own structure, a transformation candidate
$\widehat{G}$ might have to be preprocessed into a real, triangular
form with non-negative elements that is suitable for a novel,
numerically stable computation of the transformation matrices
$\widehat{U}^{\ast}$ and $\widehat{V}$ and the hyperbolic singular
values in $\widehat{\Sigma}$.
\subsection{The $\widehat{J}$-UTV factorization of $\widehat{G}$}\label{ss:2.1}
The preprocessing of $\widehat{G}$ starts with checking if
$\widehat{G}$ is diagonal, because this information will affect the
convergence criterion (see subsection~\ref{ss:5.1}) and can simplify
an optimized implementation of the $2\times 2$ HSVD\@.  Then,
$\widehat{G}$ is transformed into a real, triangular matrix
$\widehat{T}$ with non-negative elements by the $\widehat{J}$-UTV
factorization that preserves the hyperbolic singular values.  The
$\widehat{J}$-UTV factorization is similar to the
URV~\cite{Stewart-92} factorization, with the form
\begin{displaymath}
  \check{U}^{\ast}\widehat{G}\check{V}=\widehat{T},
\end{displaymath}
where $\check{U}$ is unitary, $\check{V}$ is $\widehat{J}$-unitary,
and $\widehat{T}$ is either upper or lower triangular with real,
non-negative elements.  Moreover, $\widehat{T}$ is upper triangular in
the trigonometric case, and
$\hat{t}_{11}^2\ge\hat{t}_{12}^2+\hat{t}_{22}^2$.  In the
hyperbolic case $\widehat{T}$ is either upper triangular, as
described, or lower triangular, with
$\hat{t}_{22}^2\ge\hat{t}_{21}^2+\hat{t}_{11}^2$, depending on
whether the first or the second column of $\widehat{G}$ has a greater
Frobenius norm, respectively.  The $2\times 2$ $\widehat{J}$-UTV
factorization always exists and is uniquely determined, as
demonstrated by Algorithm~\ref{a:2.1}, which, given $\widehat{G}$,
operates on a scaled matrix $\widehat{G}_0=2^s\widehat{G}$ to avoid
the possibility of floating-point overflows.  Determination of $s$ is
described in subsection~\ref{sss:2.1.1}, with the only assumption that
all (components of) the elements of $\widehat{G}$ are finite.  Assume
$\arg(0)=0$ for simplicity, and let $P_2$ be the $2\times 2$
permutation matrix
$\left[\begin{smallmatrix}0 & 1\\1 & 0\end{smallmatrix}\right]$.  Let
Boolean flags---those being only true ($\top$) or false ($\bot$)---be
denoted by small caps (e.g., \textsc{d}, \textsc{h}, and \textsc{c},
indicating whether $\widehat{G}_0$ is diagonal, whether the hyperbolic
or the trigonometric case is considered, and whether the columns of
$\widehat{G}_0$ have been swapped, respectively).

\begin{algorithm}[hbtp]
  \SetKwInput{Note}{Note}
  \leIf(\tcp*[f]{specifically, $\text{\sc d}=\top$ if $\widehat{G}_0=\mathbf{0}$}){$\widehat{G}_0$ {\rm is diagonal}}{$\text{\sc d}=\top$}{$\text{\sc d}=\bot$}
  \leIf(\tcp*[f]{$\bot$: trigonometric, \hskip-1pt$\top$: hyperbolic case}){$\widehat{J}=I_2$ {\rm or} $\widehat{J}=-I_2$}{$\text{\sc h}=\bot$}{$\text{\sc h}=\top$}
  \tcp{\hfill STAGE~$1$: column pivoting of $\widehat{G}_0$\hfill}
  \leIf{$\left\|\left[\begin{smallmatrix}(\widehat{G}_0)_{11}\\(\widehat{G}_0)_{21}\end{smallmatrix}\right]\right\|_F<\left\|\left[\begin{smallmatrix}(\widehat{G}_0)_{12}\\(\widehat{G}_0)_{22}\end{smallmatrix}\right]\right\|_F$}{$\text{\sc c}=\top$ and $\check{V}_1^{}=P_2^{}$}{$\text{\sc c}=\bot$ and $\check{V}_1^{}=I_2^{}$}
  \tcp{if $\text{\sc h}\wedge\text{\sc c}$: $\check{V}_1$ is \emph{not} $\widehat{J}$-unitary but it is canceled by $\check{V}_5$ in stage $5\text{\sc l}$}
  $\widehat{G}_1=\widehat{G}_0\check{V}_1$,\quad$\check{V}:=\check{V}_1$\tcp*[r]{start accumulating $\check{V}$}
  \tcp{\hfill STAGE~$2$: make the first column of $\widehat{G}_1$ real and non-negative\hfill}
  $\check{U}_2^{\ast}=\begin{bmatrix}e^{-\arg{(\widehat{G}_1)_{11}}}&0\\0&e^{-\arg{(\widehat{G}_1)_{21}}}\end{bmatrix}$\tcp*[r]{applied to get $(\widehat{G}_2^{})_{11}^{}=|(\widehat{G}_1^{})_{11}^{}|$ \&\ $(\widehat{G}_2^{})_{21}^{}=|(\widehat{G}_1^{})_{21}^{}|$}
  $\widehat{G}_2^{}=\check{U}_2^{\ast}\widehat{G}_1^{}$,\quad$\check{U}^{\ast}:=\check{U}_2^{\ast}$\tcp*[r]{start accumulating $\check{U}^{\ast}$}
  \tcp{\hfill STAGE~$3$: row pivoting of $\widehat{G}_2$\hfill}
  \leIf{$(\widehat{G}_2^{})_{11}^{}<(\widehat{G}_2^{})_{21}^{}$}{$\text{\sc r}=\top$ and $\check{U}_3^{\ast}=P_2^{}$}{$\text{\sc r}=\bot$ and $\check{U}_3^{\ast}=I_2^{}$}
  $\widehat{G}_3^{}=\check{U}_3^{\ast}\widehat{G}_2^{}$,\quad$\check{U}^{\ast}:=\check{U}_3^{\ast}\check{U}^{\ast}$\tcp*[r]{$(\widehat{G}_3^{})_{11}^{}\ge(\widehat{G}_3^{})_{21}^{}\ge 0$}
  \tcp{\hfill STAGE~$4$: the QR factorization of $\widehat{G}_3$\hfill}
  \eIf(\tcp*[f]{not {\sc d} implies $(\widehat{G}_3^{})_{11}^{}>0$}){$\neg\text{\sc d}$}{%
    $\tan\phi=(\widehat{G}_3^{})_{21}^{}/(\widehat{G}_3^{})_{11}^{}$,\quad$\cos\phi=1/\sqrt{1+\tan^2\phi}$\tcp*[r]{$0\le\phi\le\pi/4$}
    $\check{U}_4^{\ast}=\cos\phi\begin{bmatrix}1 & \tan\phi\\-\tan\phi & 1\end{bmatrix}$\tcp*[r]{the Givens rotation $\check{U}_4^{\ast}$ is $Q^{\ast}$ in the QR}
  }(\tcp*[f]{the QR factorization of a diagonal matrix is trivial}){%
    $\check{U}_4^{\ast}=I_2^{}$\tcp*[r]{$\phi=0$}
  }
  $\widehat{G}_4^{}=\check{U}_4^{\ast}\widehat{G}_3^{}$,\quad$\check{U}^{\ast}:=\check{U}_4^{\ast}$\tcp*[r]{$\widehat{G}_4^{}$ is upper triangular}
  \eIf(\tcp*[f]{this \emph{pipeline} ends with $\widehat{T}$ lower triangular}){$\text{\sc h}\wedge\text{\sc c}$}{%
    \tcp{\hfill STAGE~$5\text{\sc l}$: swap the columns of $\widehat{G}_4$ to cancel $\check{V}_1$\hfill}
    $\check{V}_5^{}=P_2^{}$,\quad$\widehat{G}_5^{}=\widehat{G}_4^{}\check{V}_5^{}$,\quad$\check{V}:=I_2^{}(=\check{V}\check{V}_5^{})$\tcp*[r]{since $P_2^{}=P_2^{\ast}$}
    \tcp{\hfill STAGE~$6\text{\sc l}$: swap the rows of $\widehat{G}_5$\hfill}
    $\check{U}_6^{\ast}=P_2^{}$,\quad$\widehat{G}_6^{}=\check{U}_6^{\ast}\widehat{G}_5^{}$,\quad$\check{U}^{\ast}:=\check{U}_6^{\ast}\check{U}^{\ast}$\tcp*[r]{$\widehat{G}_6^{}$ is lower triangular}
    \tcp{\hfill STAGE~$7\text{\sc l}$: make $(\widehat{G}_6^{})_{21}^{}$ real and non-negative\hfill}
    $\check{V}_7=\begin{bmatrix}e^{-\arg{(\widehat{G}_6)_{21}}}&0\\0&1\end{bmatrix}$\tcp*[r]{applied to get $(\widehat{G}_7)_{21}=|(\widehat{G}_6)_{21}|$}
    $\widehat{G}_7=\widehat{G}_6\check{V}_7$,\quad$\check{V}:=\check{V}_7$\tcp*[r]{$\check{V}_7$ is \emph{both} unitary and $\widehat{J}$-unitary}
    \tcp{\hfill STAGE~$8\text{\sc l}$: make $(\widehat{G}_7^{})_{11}^{}$ real and non-negative\hfill}
    $\check{U}_8^{\ast}=\begin{bmatrix}e^{-\arg{(\widehat{G}_7)_{11}}}&0\\0&1\end{bmatrix}$\tcp*[r]{applied to get $(\widehat{G}_8^{})_{11}^{}=|(\widehat{G}_7^{})_{11}^{}|$}
    $\widehat{G}_8^{}=\check{U}_8^{\ast}\widehat{G}_7^{}$,\quad$\check{U}^{\ast}:=\check{U}_8^{\ast}\check{U}^{\ast}$\tcp*[r]{$\widehat{T}=\widehat{G}_8^{}$}
  }(\tcp*[f]{this \emph{pipeline} ends with $\widehat{T}$ upper triangular}){%
    \tcp{\hfill STAGE~$5\text{\sc u}$: make $(\widehat{G}_4^{})_{12}^{}$ real and non-negative\hfill}
    $\check{V}_5=\begin{bmatrix}1&0\\0&e^{-\arg{(\widehat{G}_4)_{12}}}\end{bmatrix}$\tcp*[r]{applied to get $(\widehat{G}_5)_{12}=|(\widehat{G}_4)_{12}|$}
    $\widehat{G}_5^{}=\widehat{G}_4^{}\check{V}_5^{}$,\quad$\check{V}:=\check{V}\check{V}_5^{}$\tcp*[r]{$\check{V}_5$ is \emph{both} unitary and $\widehat{J}$-unitary}
    \tcp{\hfill STAGE~$6\text{\sc u}$: make $(\widehat{G}_5^{})_{22}^{}$ real and non-negative\hfill}
    $\check{U}_6^{\ast}=\begin{bmatrix}1&0\\0&e^{-\arg{(\widehat{G}_5)_{22}}}\end{bmatrix}$\tcp*[r]{applied to get $(\widehat{G}_6^{})_{22}^{}=|(\widehat{G}_5^{})_{22}^{}|$}
    $\widehat{G}_6^{}=\check{U}_6^{\ast}\widehat{G}_5^{}$,\quad$\check{U}^{\ast}:=\check{U}_6^{\ast}\check{U}^{\ast}$\tcp*[r]{$\widehat{T}=\widehat{G}_6^{}$}
  }
  \Note{All transitions $\widehat{G}_l\to\widehat{G}_{l+1}$ actually happen in-place by transforming the \emph{working matrix}.}
  \caption{The $\widehat{J}$-UTV factorization of a complex $2\times 2$ matrix $\widehat{G}_0$.}
  \label{a:2.1}
\end{algorithm}

When $\widehat{T}$ is upper triangular, Algorithm~\ref{a:2.1} is a
refinement of the URV factorization, as described in detail
in~\cite{Novakovic-20} for the trigonometric case, where
$R=\widehat{T}$ is real and non-negative, with a prescribed relation
among the magnitudes of its elements.  The $2\times 2$
$\widehat{J}$-UTV factorization guarantees that $\check{V}$ is
\emph{both} unitary and $\widehat{J}$-unitary in all cases.  When
$\widehat{G}$ is real, a multiplication of an element $\hat{g}$ of the
working matrix by $e^{-\arg{\hat{g}}}$ amounts to a multiplication by
$\sign{\hat{g}}$, without any rounding error.

In brief, Algorithm~\ref{a:2.1} consists of two possible pipelines,
having in common the first four stages.  The longer pipeline has
eight, and the shorter one six stages in total.

In a stage $\mathfrak{s}\ge 1$, the working matrix and either
$\check{U}^{\ast}$ or $\check{V}$ are transformed in place, either by
a left multiplication by $\widehat{U}_{\mathfrak{s}}^{\ast}$, or by a
right multiplication by $\widehat{V}_{\mathfrak{s}}^{}$,
respectively.  The working matrix thus progresses from the state
$\widehat{G}_{\mathfrak{s}-1}$ to $\widehat{G}_{\mathfrak{s}}$, until
it is eventually reduced to $\widehat{T}$, at which point
$\check{U}^{\ast}$ and $\check{V}$ are considered final, i.e., the
$\widehat{J}$-UTV factorization is computed.  Initially,
$\check{U}^{\ast}$ and $\check{V}$ are set to identities, and
\textsc{d} and \textsc{h} are determined.

In stage~1, if the column pivoting of $\widehat{G}_0^{}$ is needed
(i.e., its first column is of lower Frobenius norm than its second),
\textsc{c} is set to $\top$ and $\check{V}_1^{}$ to $P_2^{}$, else
they become $\bot$ and $I_2^{}$, respectively.  Note that
$\check{V}_1^{}$ is not $\widehat{J}$-unitary if \textsc{h} and
\textsc{c} are $\top$, but in that case $\check{V}_1^{}$ will be
cancelled by $P_2^{}=P_2^{\ast}$, with no intervening right
transformations, in stage~5\textsc{l}.

In stage~2, the first column of $\widehat{G}_1^{}$ is made real and
non-negative by $\check{U}_2^{\ast}$, effectively setting the elements
of the first column of $\widehat{G}_2^{}$ to the magnitudes of those
from $\widehat{G}_1^{}$.  Then, it is possible to compare them in
stage~3.  If the first is less than the second, the rows of
$\widehat{G}_2^{}$ are swapped by $\check{U}_3^{\ast}$.  In stage~4
(the final one shared among the two pipelines) the QR factorization of
$\widehat{G}_3^{}$ is performed.  If $\widehat{G}_0^{}$ was diagonal,
$\widehat{G}_3^{}$ will come out as such from the previous stage, and
the factorization is trivial.  Else, the \emph{real} Givens rotation
$\check{U}_4^{\ast}$ (i.e., $Q^{\ast}$) is computed from the elements
of the first column of $\widehat{G}_3^{}$.  Due to the row pivoting,
the rotation's angle $\phi$ lies between $0$ and $\pi/4$.  Now, $R$
from the factorization, i.e., $\widehat{G}_4^{}$, is upper triangular,
with its second column possibly complex, while
$(\widehat{G}_4^{})_{11}^{}$ is by construction real and of the
largest magnitude in $\widehat{G}_4^{}$.

The remaining stages of the longer pipeline, i.e., the one in which
\textsc{h} and \textsc{c} are $\top$, are described first.  In
stage~5\textsc{l} $\check{V}_1^{}=P_2^{}$ is cancelled by
$\check{V}_5^{}=P_2^{}$, i.e., the columns of $\widehat{G}_4^{}$ are
swapped to obtain $\widehat{G}_5^{}$, which in stage~6\textsc{l} gets
its rows swapped by $\check{U}_6^{\ast}$ to become the lower
triangular $\widehat{G}_6^{}$.   Stage~5\textsc{l} is necessary to
make the current $\check{V}$ $\widehat{J}$-unitary (identity) again.
Stage~6\textsc{l} brings the working matrix into a more ``familiar''
(i.e., lower triangular) form.  The working matrix in this pipeline
remains lower triangular until the end.  To emphasize that,
``\textsc{l}'' is appended to the stage identifiers.  To make the
working matrix real, $\check{V}_7^{}$ (obviously unitary but also
$\widehat{J}$-unitary, since
$\check{V}_7^{\ast}\widehat{J}\check{V}_7^{}=\widehat{J}$, where
$\widehat{J}$ has two opposite signs on its diagonal) is applied to it
in stage~7\textsc{l}, to turn $(\widehat{G}_6^{})_{21}^{}$ into its
magnitude.  Finally, in stage~8\textsc{l},
$(\widehat{G}_7^{})_{11}^{}$ is turned into its magnitude by
$\check{U}_8^{\ast}$.

The shorter pipeline, followed when \textsc{h} or \textsc{c} are
$\bot$, keeps the working matrix upper triangular, and thus the stage
identifiers have ``\textsc{u}'' appended to them.  Stage~5\textsc{u}
turns $(\widehat{G}_4^{})_{12}^{}$ into its magnitude by (similarly as
above, both unitary and $\widehat{J}$-unitary) $\check{V}_5^{}$, while
stage~6\textsc{u} completes the pipeline by turning
$(\widehat{G}_5^{})_{22}^{}$ into its magnitude by
$\widehat{U}_6^{\ast}$.

\looseness=-1
Figure~\ref{f:2.1} illustrates the possible progressions of a
non-diagonal complex matrix $\widehat{G}_0$ through the stages of
Algorithm~\ref{a:2.1} in both its pipelines (but only one pipeline is
taken with any given input), and as if all optional transformations
take place (the column and/or the row pivoting may not actually
happen).  A diagonal matrix admits a simplified procedure, evident
enough for its details to be omitted here for brevity.  The state of
the working matrix before the operations of each stage commence is
shown, while the transformation matrices are implied by the effects of
a particular stage.

\begin{figure}[hbt]
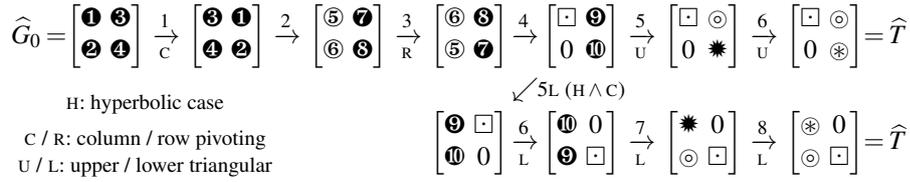
\normalsize
  \begin{displaymath}
    \begin{array}{rcccccc@{\mkern2.45mu}c@{\mkern2.45mu}ccccl}
      \widehat{G}_0\mkern-1mu=\mkern-4mu\begin{bmatrix}
        \text{\ding{182}} & \text{\ding{184}}\\
        \text{\ding{183}} & \text{\ding{185}}
      \end{bmatrix} &
      \xrightarrow[\text{\sc c}]{1} &
      \begin{bmatrix}
        \text{\ding{184}} & \text{\ding{182}}\\
        \text{\ding{185}} & \text{\ding{183}}
      \end{bmatrix} &
      \xrightarrow{2} &
      \begin{bmatrix}
        \text{\ding{176}} & \text{\ding{188}}\\
        \text{\ding{177}} & \text{\ding{189}}
      \end{bmatrix} &
      \xrightarrow[\text{\sc r}]{3} &
      \begin{bmatrix}
        \text{\ding{177}} & \text{\ding{189}}\\
        \text{\ding{176}} & \text{\ding{188}}
      \end{bmatrix} &
      \xrightarrow{4} &
      \begin{bmatrix}
        \boxdot & \text{\ding{190}}\\
        0 & \text{\ding{191}}
      \end{bmatrix} &
      \xrightarrow[\text{\sc u}]{5} &
      \begin{bmatrix}
        \boxdot & \circledcirc\\
        0 & \text{\ding{89}}
      \end{bmatrix} &
      \xrightarrow[\text{\sc u}]{6} &
      \begin{bmatrix}
        \boxdot & \circledcirc\\
        0 & \circledast
      \end{bmatrix}\mkern-4mu=\mkern-1mu\widehat{T}\\[8pt]
      \null & \null & \null & \null & \null & \null & \null & \swarrow & \mathclap{\text{\small 5{\sc l} ($\text{\sc h}\wedge\text{\sc c}$)}}\\[-4pt]
      \mathclap{\genfrac{}{}{0pt}{0}{\genfrac{}{}{0pt}{0}{\text{\small{\sc h}: hyperbolic case}}{\text{\small{\sc c} / {\sc r}: column / row pivoting}}}{\text{\small{\sc u} / {\sc l}: upper / lower triangular}}} & \null & \null & \null & \null & \null &
      \begin{bmatrix}
        \text{\ding{190}} & \boxdot\\
        \text{\ding{191}} & 0
      \end{bmatrix} &
      \xrightarrow[\text{\sc l}]{6} &
      \begin{bmatrix}
        \text{\ding{191}} & 0\\
        \text{\ding{190}} & \boxdot
      \end{bmatrix} &
      \xrightarrow[\text{\sc l}]{7} &
      \begin{bmatrix}
        \text{\ding{89}} & 0\\
        \circledcirc & \boxdot
      \end{bmatrix} &
      \xrightarrow[\text{\sc l}]{8} &
      \begin{bmatrix}
        \circledast & 0\\
        \circledcirc & \boxdot
      \end{bmatrix}\mkern-4mu=\mkern-1mu\widehat{T}
    \end{array}
  \end{displaymath}
  \caption{A schematic representation of the $\widehat{J}$-UTV
    factorization pipelines of a complex $2\times 2$ matrix
    $\widehat{G}_0$.  The upper pipeline is followed in the
    trigonometric case, or in the hyperbolic case when no column
    pivoting was performed in stage~1, resulting in an upper
    triangular $\widehat{T}$ after six transformations.   Otherwise,
    the lower pipeline is followed from stage~5 onwards, resulting in
    a lower triangular $\widehat{T}$ after eight transformations of
    the working matrix.  The symbols with a black background stand for
    the complex  elements, while those with a white background denote
    the real non-negative ones.  The circled numbers represent the
    transient values, while the framed operators are reserved for the
    final ones, such that $\boxdot^2\ge\circledcirc^2+\circledast^2$.}
  \label{f:2.1}
\end{figure}

The numerical issues of computing in the machine precision (in
Fortran) are dealt with in subsection~\ref{sss:2.1.1} for the column
norms, in subsection~\ref{sss:2.1.2} for the polar form of a complex
number, and in subsection~\ref{sss:2.1.3} for the rotation-like
transformations.  It is assumed that the floating-point arithmetic is
not trapping on an exception, that the rounding to nearest (with tie
to even) is employed, that the gradual underflow and the fused
multiply-add ($\fma$) operation are available, that
$\mathop{\mathtt{HYPOT}}(a,b)=\sqrt{a^2+b^2}$ intrinsic causes no
undue overflow, and that $\mathop{\mathtt{HYPOT}}(a,b)=0\iff a=b=0$.

Algorithm~\ref{a:2.1} is conditionally reproducible, i.e., its results
on any given input are bitwise identical in every floating-point
environment behaving as described above and employing the same
datatype, provided that the $\mathtt{HYPOT}$ intrinsic is
reproducible.
\subsubsection{Determination of the scaling parameter $s$}\label{sss:2.1.1}
For a complex number $z$, $|z|$ can overflow in floating-point
arithmetic.  The magnitude of $z$ can also pose a problem, e.g., in
computing of the norms of the columns of $\widehat{G}$ in stage~1 of
Algorithm~\ref{a:2.1} for the column pivoting in the QR
factorization.  There,
\begin{displaymath}
  \left\|\begin{bmatrix}a&b\end{bmatrix}^T\right\|_F^{}=\sqrt{|a|^2+|b|^2}=\mathop{\mathtt{HYPOT}}(\mathop{\mathtt{HYPOT}}(\RE{a},\IM{a}),\mathop{\mathtt{HYPOT}}(\RE{b},\IM{b}))=c,
\end{displaymath}
so it has to be ensured, by an appropriate prescaling of
$\widehat{G}$, that $|a|$, $|b|$, and $c$ do not overflow.  A similar
concern is addressed in the \texttt{xLARTG}
LAPACK~\cite{Anderson-et-al-99} routines, but for efficiency it is
advisable to avoid the overhead of function calls in the innermost
computational parts.  At the same time, to avoid computing with the
subnormal (components of) numbers, $\widehat{G}$ could be upscaled
when no overflow can occur.  Therefore, to efficiently mitigate as
many overflow and underflow issues as practicable, either the
following prescaling of $\widehat{G}$, and the corresponding
postscaling of the resulting $\widehat{\Sigma}$, should be employed,
or, without any scaling needed, the entire computation of the
$2\times 2$ HSVD should be performed in a floating-point datatype at
least as precise as the type of (the components of) the elements of
$\widehat{G}$, but with a wider exponent range (e.g., in the Intel's
80-bit, hardware supported extended precision datatype).

Down- and up-scalings of $\widehat{G}$ can be made (almost) exact,
without any loss of precision in the vast majority of cases, by
decreasing, or respectively increasing, the exponents of (the
components of) the elements only, i.e., by multiplying the values by
$2^s$ with $s\in\mathbb{Z}$, leaving the significands intact (except
for a subnormal value and $s<0$, when the $|s|$ lowest bits of the
significand---some of which may not be zero---perish), as explained
in~\cite[subsection~2.1]{Novakovic-20} and summarized here in Fortran
parlance.

Let
$\mathop{\eta}(a)=\mathop{\mathtt{EXPONENT}}(a)$ for a floating-point
value $a$, and let $\nu=\mathop{\eta}(\Omega)$, where
$\Omega=\mathop{\mathtt{HUGE}}(\mathtt{0D0})$ is the largest finite
\texttt{DOUBLE PRECISION} value.  For $a$ finite, let
\begin{displaymath}
  \mathop{\chi}(a)=\nu-\mathop{\eta}(\max\{|a|,\omega\})-2
\end{displaymath}
be a measure of how much, in terms of its exponent, $a$ has to be
downscaled, or how much it can be safely upscaled.  Here and in the
following, $\omega$ is the smallest non-zero positive subnormal
floating-point number\footnote{For $a=0$, $\mathop{\eta}(a)=0$ instead
of a huge negative integer, so taking $\max\{|a|,\omega\}$ filters out
such $a$.}, while
$\widehat{\omega}=\mathop{\mathtt{TINY}}(\mathtt{0D0})$ is the
smallest positive normal \texttt{DOUBLE PRECISION} value.  Two sets of
such exponent differences,
\begin{align*}
  E_{\Re}&=\{\mathop{\chi}(\RE{\hat{g}_{11}}),\mathop{\chi}(\RE{\hat{g}_{21}}),\mathop{\chi}(\RE{\hat{g}_{12}}),\mathop{\chi}(\RE{\hat{g}_{22}})\},\\
  E_{\Im}&=\{\mathop{\chi}(\IM{\hat{g}_{11}}),\mathop{\chi}(\IM{\hat{g}_{21}}),\mathop{\chi}(\IM{\hat{g}_{12}}),\mathop{\chi}(\IM{\hat{g}_{22}})\},
\end{align*}
are then formed, and the scaling parameter $s$ is taken to be the
minimum of $E_{\Re}\cup E_{\Im}$,
\begin{equation}
  s=\min\{\min{E_{\Re}},\min{E_{\Im}}\}.
  \label{e:2.1}
\end{equation}
Each component $c$ of the elements of $\widehat{G}$ is then scaled
as $c_0=\mathop{\mathtt{SCALE}}(c,s)$ to get $\widehat{G}_0$.

Backscaling of the computed hyperbolic singular values by $2^{-s}$ can
cause overflows or underflows (to the subnormal range, or even zero),
but they are unavoidable if the scaling should be transparent to the
users of the $2\times 2$ HSVD\@.  A remedy, discussed
in~\cite{Novakovic-20}, would be to keep the scales in addition to the
values to which they apply, in a separate array.  Alternatively,
computing in a wider datatype avoids both scaling and such issues, at
a cost of a slower and not easily vectorizable arithmetic.

In brief, the scaling parameter $s$ from~\eqref{e:2.1} is chosen in
such a way that:
\begin{compactenum}
\item no column norm of a scaled $2\times 2$ real or complex matrix
  can overflow, and
\item no \emph{ordinary} singular value of a scaled matrix can
  overflow~\cite[Theorem~1]{Novakovic-20}, and
\item as many subnormal (components of the) matrix elements as
  possible are brought up to the normal range, for a faster and maybe
  more accurate computation.
\end{compactenum}
The following Example~\ref{x:2.1} illustrates several reasons for and
issues with the scaling.

\begin{example}\label{x:2.1}
  Assume $\widehat{J}^{[i]}=I_2$, and consider the following real
  $\widehat{G}^{[i]}$ for $1\le i\le 3$,
  \begin{displaymath}
    \widehat{G}^{[1]}=
    \begin{bmatrix}
      \Omega & \Omega/4\\
      \Omega/8 & \Omega/2
    \end{bmatrix},\quad
    \widehat{G}^{[2]}=
    \begin{bmatrix}
      \Omega/16 & \widehat{\omega}/4\\
      \widehat{\omega}/2 & \widehat{\omega}/8
    \end{bmatrix},\quad
    \widehat{G}^{[3]}=
    \begin{bmatrix}
      \omega & 0\\
      0 & \Omega
    \end{bmatrix},
  \end{displaymath}
  alongside their scaled versions $\widehat{G}_0^{[i]}=2^{s^{[i]}}\widehat{G}^{[i]}$,
  \begin{displaymath}
    \widehat{G}_0^{[1]}=
    \begin{bmatrix}
      \Omega/4 & \Omega/16\\
      \Omega/32 & \Omega/8
    \end{bmatrix},\quad
    \widehat{G}_0^{[2]}=
    \begin{bmatrix}
      \Omega/4 & \widehat{\omega}\\
      2\widehat{\omega} & \widehat{\omega}/2
    \end{bmatrix},\quad
    \widehat{G}_0^{[3]}=
    \begin{bmatrix}
      \omega/4=0 & 0\\
      0 & \Omega/4
    \end{bmatrix},
  \end{displaymath}
  with the scaling parameters, computed from~\eqref{e:2.1}, being
  $s^{[1]}=s^{[3]}=-2$ and $s^{[2]}=2$.

  The singular values of $\widehat{G}^{[1]}$, computed symbolically,
  are $\Omega(\sqrt{145}\pm 5)/16$, and obviously the larger one would
  overflow without the scaling.  In $\widehat{G}^{[2]}$ there are
  three subnormal values.  Two of them (the off-diagonal ones) get
  raised into the normal range, while the diagonal one stays
  subnormal, after the scaling.  Finally, $\widehat{G}^{[3]}$ is a
  pathological case of a non-singular matrix that is as
  ill-conditioned as possible, which becomes singular after the
  scaling, since $\omega/4=0$ in the rounding-to-nearest mode.
\end{example}
\subsubsection{Accurate floating-point computation of the polar form of a complex number}\label{sss:2.1.2}
Expressing $z$ as $|z|e^{\mathrm{i}\arg{z}}$ requires a reliable way
of computing $e^{\mathrm{i}\arg{z}}$, where
\begin{displaymath}
  e^{\mathrm{i}\arg{z}}=\cos(\arg{z})+\mathrm{i}\cdot\sin(\arg{z}).
\end{displaymath}
One such method, presented in~\cite[eq.~(1)]{Novakovic-20}, is
summarized as follows.  It is assumed that $z$ has already been scaled
as described in subsection~\ref{sss:2.1.1}, so $|z|$ cannot overflow.

Let $\mathtt{MIN}$ be a floating-point minimum function\footnote{For
example, the \textbf{minimumNumber} operation of the IEEE 754-2019
standard~\cite[section~9.6]{IEEE-754-2019}.}, with a property that if
its first argument is a $\mathtt{NaN}$ and the second is not, the
result is the second argument.  For the $\mathtt{IEEE\_MIN\_NUM}$
Fortran~2018~\cite{Fortran-18} intrinsic such a behavior is
guaranteed, but is also present with the $\mathtt{MIN}$ intrinsic of
the recent Intel Fortran compilers, and several others.  The same
property will be required for a maximum function $\mathtt{MAX}$ later
on.  Define
\begin{displaymath}
  \cos(\arg{z})=\mathop{\mathtt{MIN}}\left(\frac{|\RE{z}|}{|z|},1\right)\cdot\sign(\RE{z}),\quad
  \sin(\arg{z})=\frac{\IM{z}}{\max\{|z|,\omega\}}.
\end{displaymath}
When $|z|=0$, $\mathtt{MIN}$ ensures the correct value of
$\cos(\arg{z})$, because $\mathop{\mathtt{MIN}}(0/0,1)=1$, while in
that case taking the maximum in the denominator of $\sin(\arg{z})$
makes the whole quotient well defined, since $0/\omega=0$.  When
$|z|>0$, it holds $\max\{|z|,\omega\}=|z|$.
\subsubsection{Performing the rotation-like transformations in floating-point}\label{sss:2.1.3}
Some real-valued quantities throughout the paper are obtained by the
expressions of the form $a\cdot b+c$, and should be computed by a
single $\fma$, i.e., with a single rounding.  One such example is
$\cos\phi=1/\sqrt{\tan\phi\cdot\tan\phi+1}$ in stage~4 of
Algorithm~\ref{a:2.1}.

Regrettably, there is no widespread hardware support for the correctly
rounded reciprocal square root (i.e., $1/\sqrt{x}$) operation.
Therefore, the computation of a cosine involves at least three
roundings: from the fused multiply-add operation to obtain the
argument of the square root, from the square root itself, and from
taking the reciprocal value.  Instead of the cosines, the respective
secants can be computed with two roundings, without taking the
reciprocals.  Then, in all affected formulas the multiplications by
the cosines can be replaced by the respective divisions by the
secants.  Such an approach is slower than the usual one, but more
accurate in the worst case.

There is no standardized way to accurately compute $d=a\cdot b+c$ with
some or all values being complex.  However, the real fused
multiply-add operation can be employed as in the
CUDA~\cite[\texttt{cuComplex.h} header]{NVidia-19} implementation of
the complex arithmetic for an efficient, accurate, and reproducible
computation of $d$ as
\begin{displaymath}
  \begin{aligned}
    \RE(d)&=\fma(\RE(a),\RE(b),\fma(-\IM(a),\IM(b),\RE(c))),\\
    \IM(d)&=\fma(\RE(a),\IM(b),\fma(\hphantom{-}\IM(a),\RE(b),\IM(c))),
  \end{aligned}
\end{displaymath}
when $a$, $b$, and $c$ are complex.  Otherwise, when $a$ is real,
\begin{displaymath}
  \RE(d)=\fma(a,\RE(b),\RE(c)),\quad
  \IM(d)=\fma(a,\IM(b),\IM(c)),
\end{displaymath}
and when $a$ is non-zero and purely imaginary,
\begin{displaymath}
  \RE(d)=\fma(-\IM(a),\IM(b),\RE(c)),\quad
  \IM(d)=\fma(\IM(a),\RE(b),\IM(c)).
\end{displaymath}
Such operations can be implemented explicitly by the
$\mathtt{IEEE\_FMA}$ Fortran~2018~\cite{Fortran-18} intrinsic, or
implicitly, by relying on the compiler to emit the appropriate $\fma$
instructions.  By an abuse of notation, $\fma(a,b,c)$ in the following
stand for both the real-valued and the above complex-valued
operations, depending on the context.

With $a$ and $b$ complex, a multiplication $a\cdot b$ can be expressed
as $\fma(a,b,c)$ with $c=0$ and implemented as such, in a reproducible
way, but simplified by converting all real $\fma$ operations involving
a component of $c$ to real multiplications.

A plane rotation, if it is to be applied from the left, can be written
as (see, e.g.,~\cite{Drmac-97})
\begin{displaymath}
  \begin{bmatrix}
    \hphantom{\mp}\cos\phi & \pm\sin\phi\\
    \mp\sin\phi & \hphantom{\pm}\cos\phi
  \end{bmatrix}=
  \cos\phi\cdot
  \begin{bmatrix}
    1 & \pm\tan\phi\\
    \mp\tan\phi & 1    
  \end{bmatrix}=
  C\cdot T,
\end{displaymath}
and similarly, if a plane (i.e., trigonometric) rotation is to be
applied from the right,
\begin{displaymath}
  T\cdot C=
  \begin{bmatrix}
    1 & \pm\tan\phi\\
    \mp\tan\phi & 1    
  \end{bmatrix}
  \cdot\cos\phi=
  \begin{bmatrix}
    \hphantom{\mp}\cos\phi & \pm\sin\phi\\
    \mp\sin\phi & \hphantom{\pm}\cos\phi
  \end{bmatrix}.
\end{displaymath}
A multiplication by $T$ can be realized by a single $\fma$ per an
element of the result, while the subsequent scaling by $C$ can be
converted to divisions by the corresponding secants.  A similar
factorization holds for a hyperbolic rotation, e.g., from the right,
\begin{displaymath}
  T\cdot C=
  \begin{bmatrix}
    1 & \tanh\phi\\
    \tanh\phi & 1    
  \end{bmatrix}
  \cdot\cosh\phi=
  \begin{bmatrix}
    \cosh\phi & \sinh\phi\\
    \sinh\phi & \cosh\phi
  \end{bmatrix}.
\end{displaymath}
\subsection{The HSVD of $\widehat{T}$}\label{ss:2.2}
Given $\widehat{T}$ from Algorithm~\ref{a:2.1}, the HSVD of
$\widehat{T}$,
$\widetilde{U}^{\ast}\widehat{T}\widetilde{V}^{}=\widetilde{\Sigma}$,
now remains to be found.

\looseness=-1
First,
$\widetilde{U}^{\prime\ast}\widehat{T}\widetilde{V}'=\widetilde{\Sigma}'$
is computed, where (despite the notation) all matrices are real,
$\widetilde{\Sigma}'$ is diagonal with non-negative elements,
$\widetilde{U}^{\prime\ast}$ is orthogonal, and is sought in the form
of a plane rotation, while $\widetilde{V}'$ is
$\widehat{J}$-orthogonal, and is sought in the form of a plane
rotation when $\widehat{J}=I_2^{}$ or $\widehat{J}=-I_2^{}$, and in
the form of a hyperbolic rotation otherwise.

If $\widehat{J}\ne-I_2^{}$, $\widetilde{\Sigma}'$ is equal to the
matrix $\widetilde{\Sigma}$ of the hyperbolic singular values of
$\widehat{T}$.  Else,
$\widetilde{\Sigma}'=P_2^{}\widetilde{\Sigma}P_2^{\ast}$, i.e., the
diagonal elements of $\widetilde{\Sigma}'$ are in the order opposite
to the one prescribed by Definition~\ref{d:1.1}, and thus have to be
swapped to get
$\widetilde{\Sigma}=P_2^{\ast}\widetilde{\Sigma}'P_2^{}$.  In the
latter case, when $\widehat{J}=-I_2^{}$, let $\widetilde{U}^{\ast}$
and $\widetilde{V}$ be $P_2^{\ast}\widetilde{U}^{\prime\ast}$ and
$\widetilde{V}'P_2^{}$, respectively; else, let
$\widetilde{U}^{\ast}=\widetilde{U}^{\prime\ast}$ and
$\widetilde{V}=\widetilde{V}'$.  Note that $\widetilde{\Sigma}$ holds
the scaled hyperbolic singular values of $\widehat{G}$.

\looseness=-1
There are four non-disjoint cases for $\widehat{T}$ and $\widehat{J}$,
to be considered in the following order, stopping at, and proceeding
as in, the first case for which its condition is satisfied:
\begin{compactenum}
\item if $\widehat{T}$ is already diagonal then let
  $\widetilde{U}^{\prime\ast}=\widetilde{V}'=I_2^{}$ and
  $\widetilde{\Sigma}'=\widehat{T}$, else
\item if $\widehat{J}=I_2$ or $\widehat{J}=-I_2$, $\widehat{T}$ is
  upper triangular and proceed as in subsection~\ref{sss:2.2.1}, else
\item if $\widehat{T}$ is upper triangular, then proceed as in
  subsection~\ref{sss:2.2.2}, else
\item $\widehat{T}$ is lower triangular, and proceed as in
  subsection~\ref{sss:2.2.3}.
\end{compactenum}
The first case above corresponds to $\text{\sc d}=\top$, the second
one to $\text{\sc h}=\bot$, the third one to
$\text{\sc h}\wedge\neg\text{\sc c}=\top$, and the last one to
$\text{\sc h}\wedge\text{\sc c}=\top$, in terms of the states of
Algorithm~\ref{a:2.1}.

In fact, instead of $\widetilde{\Sigma}'$, the backscaled
$\widehat{\Sigma}'=2^{-s}\widetilde{\Sigma}'$ is directly computed for
stability, i.e., the hyperbolic singular values of $\widehat{G}$,
instead of $\widehat{T}$, are finally obtained by the backscaling
procedure described in a separate paragraph and Algorithm~\ref{a:2.2}
below.

\paragraph{Backscaling.}
A backscaling routine from Algorithm~\ref{a:2.2} is used in
Algorithms~\ref{a:2.3}, \ref{a:2.4}, and \ref{a:2.5}.
Algorithm~\ref{a:2.2} distributes the scale $2^{-s}$, with $s$
from~\eqref{e:2.1}, among one of the values to which it should be
applied and the given factor $f$, $1\le f\le\Omega/4$.  The values
$d_1^{}$ and $d_2^{}$ correspond to the diagonal elements of
$\widehat{T}$, in some order.  From~\eqref{e:2.8}, \eqref{e:2.14}, and
\eqref{e:2.20}, it follows that the scaled hyperbolic singular values
are of the form $d_1^{}f$ and $d_2^{}/f$, for a certain $f$.  The idea
behind Algorithm~\ref{a:2.2} is to upscale $f$ (since $s\ge -2$, it
cannot be by more than fourfold) or downscale it as much as possible,
while keeping it above the underflow threshold $\widehat{\omega}$, to
$f'$.  Any remaining backscaling factor is applied to $d_1^{}$ before
multiplying it by $f'$ to get $d_1'$, while $d_2^{}$ is divided by $f$
and the result is fully backscaled by $2^{-s}$ to obtain $d_2'$.
With such a distribution of $2^{-s}$ among $d_1^{}$ and $f$, neither
of them should lose accuracy in an intermediate computation leading to
$d_1'$ by being needlessly pushed down to the subnormal range, what
could otherwise happen by multiplying any of them by the full
backscaling factor.

\begin{algorithm}[hbtp]
  \SetKwFunction{Backscale}{backscale}
  \textbf{subroutine} \Backscale{$s$, $d_1^{}$, $d_2^{}$, $f$, $d_1'$, $d_2'$}\;
  \KwIn{$s$ from~\eqref{e:2.1}, $d_1\ge 0$, $d_2\ge 0$, $1\le f\le\Omega/4$}
  \KwOut{$d_1'\ge 0$, $d_2'\ge 0$}
  \KwData{$\mu=\mathop{\eta}(\widehat{\omega})$}
  $\Delta=\mathop{\eta}(f)-\mu$\tcp*[r]{exponents' distance between $f$ and the underflow threshold}
  \eIf(\tcp*[f]{partial backscaling of $f$}){$s>\Delta$}{$f'=2^{-\Delta}f$,\quad$\xi=s-\Delta$\;}(\tcp*[f]{full backscaling of $f$}){$f'=2^{-s}f$,\quad$\xi=0$\;}
  $d_1'=(2^{-\xi}d_1^{})f'$,\quad$d_2'=2^{-s}(d_2^{}/f)$\;
  \caption{A safe backscaling of the hyperbolic singular values.}
  \label{a:2.2}
\end{algorithm}

For example, let $f=1'$, the immediate floating-point successor of
unity, and $2^{-s}=\widehat{\omega}/2$.  Then, $2^{-s}f$ would cause
the least significant bit of $f$ to perish, with the result being
subnormal.  It might also happen that $2^{-s}d_1^{}$ becomes
subnormal.  Instead, with Algorithm~\ref{a:2.2}, $f'$ would be the
immediate successor of $\widehat{\omega}$, so no accuracy would be
lost and normality of the result would be preserved, and $\xi=1$,
i.e., $2^{-\xi}d_1^{}=d_1^{}/2$.

Let in the following $\alpha$ be
$\mathop{\mathtt{SQRT}}(\Omega)\approx 1.34078079299425956\cdot 10^{154}$
(when computing in \texttt{DOUBLE PRECISION}), corresponding to
$\sqrt{\mathtt{DBL\_MAX}}$ parameter from~\cite{Novakovic-20}.
\subsubsection{Trigonometric case with $\widehat{T}$ upper triangular}\label{sss:2.2.1}
This special case of the ordinary SVD and has been thoroughly covered
in~\cite{Novakovic-20} (with the trigonometric angles of the opposite
signs), but is summarized here for completeness.  Many techniques to
be introduced here are also used in the hyperbolic case.

\looseness=-1
The diagonalization requirement
$\widetilde{U}^{\prime\ast}\widehat{T}\widetilde{V}'=\widetilde{\Sigma}'$
can be expressed as
\begin{displaymath}
  \begin{bmatrix}
    \hphantom{-}\cos\varphi & \sin\varphi\\
    -\sin\varphi & \cos\varphi
  \end{bmatrix}
  \begin{bmatrix}
    \hat{t}_{11}^{} & \hat{t}_{12}^{}\\
    0 & \hat{t}_{22}^{}
  \end{bmatrix}
  \begin{bmatrix}
    \cos\psi & -\sin\psi\\
    \sin\psi & \hphantom{-}\cos\psi
  \end{bmatrix}=
  \begin{bmatrix}
    \tilde{\sigma}_{11}' & 0\\
    0 & \tilde{\sigma}_{22}'
  \end{bmatrix},
\end{displaymath}
where the order of the two matrix multiplications is arbitrary.  Let
such order be fixed to
$(\widetilde{U}^{\prime\ast}\widehat{T})\widetilde{V}'$.  By observing
that $\hat{t}_{11}^{}\ge\max\{\hat{t}_{12}^{},\hat{t}_{22}^{}\}\ge 0$
and $\hat{t}_{11}^{}>0$, and by restricting the ranges for $\varphi$
and $\psi$ such that $\cos\varphi>0$ and $\cos\psi>0$, it follows that
both sides can be scaled by $1/(\hat{t}_{11}^{}\cos\varphi\cos\psi)$
to obtain a system of equations in $\tan\varphi$ and $\tan\psi$ as
\begin{equation}
  \begin{bmatrix}
    1 & \tan\varphi\\
    -\tan\varphi & 1
  \end{bmatrix}
  \begin{bmatrix}
    1 & x\\
    0 & y
  \end{bmatrix}
  \begin{bmatrix}
    1 & -\tan\psi\\
    \tan\psi & 1
  \end{bmatrix}=
  \begin{bmatrix}
    \tilde{\sigma}_{11}'' & 0\\
    0 & \tilde{\sigma}_{22}''
  \end{bmatrix},
  \label{e:2.2}
\end{equation}
where $0\le\min\{x,y\}\le\max\{x,y\}\le 1$.  This \emph{crucial}
conversion of one arbitrary value to a constant
($\hat{t}_{11}^{}\to 1$), enabled by the special form of
$\widehat{T}$, is the key difference from the standard derivation of
the Kogbetliantz formulas used in, e.g., \texttt{xLASV2} LAPACK
routines, that makes the ones to be derived here simpler but still
highly accurate~\cite{Novakovic-20}.

Equating the off-diagonal elements of the left and the right hand side
of~\eqref{e:2.2} gives the annihilation conditions
\begin{equation}
  x+y\tan\varphi-\tan\psi=0=\tan\psi(y-x\tan\varphi)-\tan\varphi,
  \label{e:2.3}
\end{equation}
from which $\tan\psi$ can be expressed as
\begin{equation}
  x+y\tan\varphi=\tan\psi=\frac{\tan\varphi}{y-x\tan\varphi},
  \label{e:2.4}
\end{equation}
where the right hand side is not defined if and only if
$y-x\tan\varphi=0$, what would imply, from~\eqref{e:2.3}, that
$\tan\varphi=0$ and thus $y=0$, so $\tan\psi=x$ and the decomposition
of $\widehat{T}$ is complete.  In all other cases,~\eqref{e:2.4} leads
to a quadratic equation in $\tan\varphi$,
\begin{displaymath}
  xy(1-\tan^2\varphi)=(x^2-y^2+1)\tan\varphi,
\end{displaymath}
or, with the terms rearranged,
\begin{displaymath}
  0\le\frac{xy}{x^2-y^2+1}=\frac{\tan\varphi}{1-\tan^2\varphi}=\frac{1}{2}\tan(2\varphi),
\end{displaymath}
where $\tan^2\varphi=1$ would imply $x^2-y^2+1=0$, what cannot happen
since $0\le y\le 1$ and $x>0$ (due to the assumption that
$\widehat{T}$ is not diagonal).  Thus, $0\le\tan(2\varphi)<\infty$,
\begin{equation}
  \tan(2\varphi)=\frac{2xy}{(x-y)(x+y)+1},
  \label{e:2.5}
\end{equation}
and $0\le\tan\varphi<1$.  Note that $y=0$ implies $\tan(2\varphi)=0$,
so $\tan\varphi=0$ and no special handling, as above, of this case is
actually required.  Also, all squares of the elements of the scaled
$\widehat{T}$ have vanished from the denominator in~\eqref{e:2.5}.
Then,
\begin{equation}
  \tan\varphi=\frac{\tan(2\varphi)}{1+\sqrt{1+\tan^2(2\varphi)}},\quad
  \sec\varphi=\sqrt{1+\tan^2\varphi},\quad
  \cos\varphi=1/\sec\varphi,
  \label{e:2.6}
\end{equation}
while, from~\eqref{e:2.4},
\begin{equation}
  \tan\psi=x+y\tan\varphi,\quad
  \sec\psi=\sqrt{1+\tan^2\psi},\quad
  \cos\psi=1/\sec\psi.
  \label{e:2.7}
\end{equation}

The following result has been proven as a part of~\cite[Theorem~1]{Novakovic-20}.
\begin{theorem}\label{t:2.1}
  It holds $\sqrt{2}\ge\tan\psi\ge\tan\varphi\ge 0$ and
  $\tilde{\sigma}_{11}'\ge\tilde{\sigma}_{22}'\ge 0$, where
  \begin{equation}
    \tilde{\sigma}_{11}'=\frac{\sec\psi}{\sec\varphi}\hat{t}_{11}^{}=\frac{\cos\varphi}{\cos\psi}\hat{t}_{11}^{},\quad
    \tilde{\sigma}_{22}'=\frac{\sec\varphi}{\sec\psi}\hat{t}_{22}^{}=\frac{\cos\psi}{\cos\varphi}\hat{t}_{22}^{}.
    \label{e:2.8}
  \end{equation}
\end{theorem}

From Theorem~\ref{t:2.1} it follows that the computed singular values
of $\widehat{T}$ are ordered descendingly.  When $\widehat{J}=-I_2$,
they have to be swapped, as described in subsection~\ref{ss:2.2}.

Algorithm~\ref{a:2.3} stably computes the derived quantities in
floating-point and works also for a diagonal $\widehat{T}$.  Bounding
$\tan(2\varphi)$ from above by $\alpha$ is necessary to avoid a
possible overflow of the argument of the square root in computing
$\tan\varphi$, without resorting to $\mathtt{HYPOT}$ or losing
accuracy of the result (see~\cite[subsection~2.3]{Novakovic-20}).  The
expression $2xy$ is computed as $(2a)b$, where $a=\min\{x,y\}$ and
$b=\max\{x,y\}$, to prevent the avoidable underflows of $xy$ (e.g.,
when $x$ is just above the underflow threshold and $y\approx 1/2$).
No division $a/b$ should be replaced by $a(1/b)$, nor square root
inaccurately approximated.  Algorithm~\ref{a:2.3} is reproducible if
a method of computing $1/\sqrt{x}$ is fixed.

\begin{algorithm}[hbtp]
  $x=\mathop{\mathtt{MAX}}(\hat{t}_{12}/\hat{t}_{11},0)$,\quad$y=\mathop{\mathtt{MAX}}(\hat{t}_{22}/\hat{t}_{11},0)$\tcp*[r]{\eqref{e:2.2} ($\mathtt{MAX}$ handles $\widehat{T}=\mathbf{0}$ case)}
  $\displaystyle\tan(2\varphi)=\min\left\{\mathop{\mathtt{MAX}}\left(\frac{(2\min\{x,y\})\max\{x,y\}}{\fma(x-y,x+y,1)},0\right),\alpha\right\}$\tcp*[r]{\eqref{e:2.5}}
  $\displaystyle\tan\varphi=\frac{\tan(2\varphi)}{1+\sqrt{\fma(\tan(2\varphi),\tan(2\varphi),1)}}$,\quad$\tan\psi=\fma(y,\tan\varphi,x)$\tcp*[r]{{\rm\eqref{e:2.6}--\eqref{e:2.7}}}
  \eIf{{\rm an accurate enough (e.g., correctly rounded) $\mathop{\mathtt{rsqrt}}(x)=1/\sqrt{x}$ function is available}}{%
    $\cos\varphi=\mathop{\mathtt{rsqrt}}(\fma(\tan\varphi,\tan\varphi,1))$\tcp*[r]{\eqref{e:2.6}}
    $\cos\psi=\mathop{\mathtt{rsqrt}}(\fma(\tan\psi,\tan\psi,1))$\tcp*[r]{\eqref{e:2.7}}
    \Backscale{$s$, $\hat{t}_{11}^{}$, $\hat{t}_{22}^{}$, $\cos\varphi/\cos\psi$, $\hat{\sigma}_{11}'$, $\hat{\sigma}_{22}'$}\tcp*[r]{\eqref{e:2.8} \&\ Algorithm~\ref{a:2.2}}
  }(\tcp*[f]{usually, $\mathop{\mathtt{rsqrt}}(x)$ is not available}){%
    $\sec\varphi=\sqrt{\fma(\tan\varphi,\tan\varphi,1)}$,\quad$\cos\varphi=1/\sec\varphi$\tcp*[r]{\eqref{e:2.6}}
    $\sec\psi=\sqrt{\fma(\tan\psi,\tan\psi,1)}$,\quad$\cos\psi=1/\sec\psi$\tcp*[r]{\eqref{e:2.7}}
    \Backscale{$s$, $\hat{t}_{11}^{}$, $\hat{t}_{22}^{}$, $\sec\psi/\sec\varphi$, $\hat{\sigma}_{11}'$, $\hat{\sigma}_{22}'$}\tcp*[r]{\eqref{e:2.8} \&\ Algorithm~\ref{a:2.2}}
  }
  \caption{The SVD of an upper triangular $\widehat{T}$ in \texttt{DOUBLE PRECISION}.}
  \label{a:2.3}
\end{algorithm}
\subsubsection{Hyperbolic case with $\widehat{T}$ upper triangular}\label{sss:2.2.2}
The diagonalization requirement
$\widetilde{U}^{\prime\ast}\widehat{T}\widetilde{V}'=\widetilde{\Sigma}'$,
scaled by $1/(\hat{t}_{11}\cos\varphi\cosh\psi)>0$ in the hyperbolic
case for $\widehat{T}$ upper triangular, is
\begin{equation}
  \begin{bmatrix}
    1 & \tan\varphi\\
    -\tan\varphi & 1
  \end{bmatrix}
  \begin{bmatrix}
    1 & x\\
    0 & y
  \end{bmatrix}
  \begin{bmatrix}
    1 & \tanh\psi\\
    \tanh\psi & 1
  \end{bmatrix}=
  \begin{bmatrix}
    \tilde{\sigma}_{11}'' & 0\\
    0 & \tilde{\sigma}_{22}''
  \end{bmatrix},
  \label{e:2.9}
\end{equation}
from which, similarly to the trigonometric case, the annihilation
conditions follow as
\begin{equation}
  x+y\tan\varphi+\tanh\psi=0=\tanh\psi(y-x\tan\varphi)-\tan\varphi.
  \label{e:2.10}
\end{equation}
If $y=0$ then $\tan\varphi=0$ and $\tanh\psi=-x$
satisfy~\eqref{e:2.10}, so the decomposition of $\widehat{T}$ is
complete, unless $x=1$ and therefore $\tanh\psi=-1$, what is
impossible.  Thus, the HSVD is \emph{not defined} when the columns of
$\widehat{T}$ (equivalently, of $\widehat{G}$) are identical.

In the remaining cases, with $y>0$, $\tanh\psi$ can be
expressed as
\begin{equation}
  x+y\tan\varphi=-\tanh\psi=\frac{\tan\varphi}{x\tan\varphi-y},
  \label{e:2.11}
\end{equation}
what gives, in a manner and with caveats similar to the trigonometric
case,
\begin{equation}
  -\infty<\tan(2\varphi)=\frac{-2xy}{(y-x)(y+x)+1}\le 0,
  \label{e:2.12}
\end{equation}
and the remaining functions of $\varphi$ are computed as
in~\eqref{e:2.6}.  From~\eqref{e:2.11} it follows
\begin{equation}
  \tanh\psi=-(x+y\tan\varphi),\quad
  \sech\psi=\sqrt{1-\tanh^2\psi},\quad
  \cosh\psi=1/\sech\psi.
  \label{e:2.13}
\end{equation}

Theorem~\ref{t:2.2} is an analogon of Theorem~\ref{t:2.1} in this
hyperbolic case.

\begin{theorem}\label{t:2.2}
  If, in~\eqref{e:2.9}, $x\ne 1$, then $|\tanh\psi|<1$ and
  \begin{equation}
    \tilde{\sigma}_{11}'=\frac{\sech\psi}{\sec\varphi}\hat{t}_{11}^{}=\frac{\cos\varphi}{\cosh\psi}\hat{t}_{11}^{},\quad
    \tilde{\sigma}_{22}'=\frac{\sec\varphi}{\sech\psi}\hat{t}_{22}^{}=\frac{\cosh\psi}{\cos\varphi}\hat{t}_{22}^{}.
    \label{e:2.14}
  \end{equation}
\end{theorem}
\begin{proof}
  From~\eqref{e:2.12}, $-1<\tan\varphi\le 0$, and from~\eqref{e:2.13},
  $|\tanh\psi|=|x+y\tan\varphi|$.  With $x$ fixed, $|\tanh\psi|$
  attains the maximum value for $y\tan\varphi$ being either the
  smallest, i.e., when approaching $-1$ from the right, or the largest
  possible, i.e., zero.  In the first case, $|\tanh\psi|<|x-1|\le 1$.
  In the second, $|\tanh\psi|=|x|=x$, so if $x\ne 1$ then
  $|\tanh\psi|<1$.

  Now~\eqref{e:2.14} is proven.  From~\eqref{e:2.9} and~\eqref{e:2.13}
  it follows
  \begin{align*}
    \tilde{\sigma}_{11}''&=1+(x+y\tan\varphi)\tanh\psi=1-\tanh^2\psi={\sech}^2\psi,\\
    \tilde{\sigma}_{22}''&=y-x\tan\varphi-\tan\varphi\tanh\psi=y-\tan\varphi(x+\tanh\psi)=y+y\tan^2\varphi=y\sec^2\varphi,
  \end{align*}
  what, after multiplying both equations by
  $\hat{t}_{11}^{}\cos\varphi\cosh\psi$, gives~\eqref{e:2.14}.
  \qed%
\end{proof}

Algorithm~\ref{a:2.4}, similarly to Algorithm~\ref{a:2.3}, computes
reproducibly and as stably as practicable the HSVD of an upper
triangular (including diagonal) $\widehat{T}$ when
$\widehat{J}\ne\pm I_2$.

\begin{algorithm}[hbtp]
  $x=\mathop{\mathtt{MAX}}(\hat{t}_{12}/\hat{t}_{11},0)$,\quad$y=\mathop{\mathtt{MAX}}(\hat{t}_{22}/\hat{t}_{11},0)$\tcp*[r]{\eqref{e:2.9} ($\mathtt{MAX}$ handles $\widehat{T}=\mathbf{0}$ case)}
  $\displaystyle\tan(2\varphi)=-\min\left\{\mathop{\mathtt{MAX}}\left(\frac{(2\min\{x,y\})\max\{x,y\}}{\fma(y-x,y+x,1)},0\right),\alpha\right\}$\tcp*[r]{\eqref{e:2.12}}
  $\displaystyle\tan\varphi=\frac{\tan(2\varphi)}{1+\sqrt{\fma(\tan(2\varphi),\tan(2\varphi),1)}}$,\quad$\tanh\psi=-\fma(y,\tan\varphi,x)$\tcp*[r]{\eqref{e:2.6}~\&~\eqref{e:2.13}}
  \lIf(\tcp*[f]{$|\tanh\psi|$ is unsafe if it is too close to unity}){$|\tanh\psi|\ge\upsilon$}{\KwRet{$\bot$}}
  \eIf{{\rm an accurate enough (e.g., correctly rounded) $\mathop{\mathtt{rsqrt}}(x)=1/\sqrt{x}$ function is available}}{%
    $\cos\varphi=\mathop{\mathtt{rsqrt}}(\fma(\tan\varphi,\tan\varphi,1))$\tcp*[r]{\eqref{e:2.6}}
    $\cosh\psi=\mathop{\mathtt{rsqrt}}(\fma(-\tanh\psi,\tanh\psi,1))$\tcp*[r]{\eqref{e:2.13}}
    \Backscale{$s$, $\hat{t}_{22}^{}$, $\hat{t}_{11}^{}$, $\cosh\psi/\cos\varphi$, $\hat{\sigma}_{22}'$, $\hat{\sigma}_{11}'$}\tcp*[r]{\eqref{e:2.14} \&\ Algorithm~\ref{a:2.2}}
  }(\tcp*[f]{usually, $\mathop{\mathtt{rsqrt}}(x)$ is not available}){%
    $\sec\varphi=\sqrt{\fma(\tan\varphi,\tan\varphi,1)}$,\quad$\cos\varphi=1/\sec\varphi$\tcp*[r]{\eqref{e:2.6}}
    $\sech\psi=\sqrt{\fma(-\tanh\psi,\tanh\psi,1)}$,\quad$\cosh\psi=1/\sech\psi$\tcp*[r]{\eqref{e:2.13}}
    \Backscale{$s$, $\hat{t}_{22}^{}$, $\hat{t}_{11}^{}$, $\sec\varphi/\sech\psi$, $\hat{\sigma}_{22}'$, $\hat{\sigma}_{11}'$}\tcp*[r]{\eqref{e:2.14} \&\ Algorithm~\ref{a:2.2}}
  }
  \KwRet{$\top$}\tcp*[r]{the hyperbolic transformation is defined and safe}
  \caption{The HSVD of an upper triangular $\widehat{T}$ in \texttt{DOUBLE PRECISION}.}
  \label{a:2.4}
\end{algorithm}

Here and in Algorithm~\ref{a:2.5} from subsection~\ref{sss:2.2.3} a
safety parameter $\upsilon\lesssim 1$ is introduced that will be fully
explained in subsection~\ref{ss:2.3}.  For now, assume $\upsilon=1$.
\subsubsection{Hyperbolic case with $\widehat{T}$ lower triangular}\label{sss:2.2.3}
The diagonalization requirement
$\widetilde{U}^{\prime\ast}\widehat{T}\widetilde{V}'=\widetilde{\Sigma}'$,
scaled by $1/(\hat{t}_{22}\cos\varphi\cosh\psi)>0$ in the hyperbolic
case for $\widehat{T}$ lower triangular, is
\begin{equation}
  \begin{bmatrix}
    1 & \tan\varphi\\
    -\tan\varphi & 1
  \end{bmatrix}
  \begin{bmatrix}
    y & 0\\
    x & 1
  \end{bmatrix}
  \begin{bmatrix}
    1 & \tanh\psi\\
    \tanh\psi & 1
  \end{bmatrix}=
  \begin{bmatrix}
    \tilde{\sigma}_{11}'' & 0\\
    0 & \tilde{\sigma}_{22}''
  \end{bmatrix},
  \label{e:2.15}
\end{equation}
from which the annihilation conditions follow as
\begin{equation}
  \tanh\psi(y+x\tan\varphi)+\tan\varphi=0=x-y\tan\varphi+\tanh\psi.
  \label{e:2.16}
\end{equation}
If $y=0$ then $\tan\varphi=0$ and $\tanh\psi=-x$
satisfy~\eqref{e:2.16}, so the decomposition of $\widehat{T}$ is
complete, unless $x=1$ and thus $\tanh\psi=-1$, what is impossible.
The HSVD of $\widehat{T}$ is \emph{not defined}, same as in
subsection~\ref{sss:2.2.2}, when the columns of $\widehat{T}$ (i.e.,
$\widehat{G}$) are identical.

In the remaining cases, with $y>0$, $\tanh\psi$ can be
expressed as
\begin{equation}
  x-y\tan\varphi=-\tanh\psi=\frac{\tan\varphi}{y+x\tan\varphi},
  \label{e:2.17}
\end{equation}
what gives, similarly to the other hyperbolic case,
\begin{equation}
  0\le\tan(2\varphi)=\frac{2xy}{(y-x)(y+x)+1}<\infty,
  \label{e:2.18}
\end{equation}
and the remaining functions of $\varphi$ are computed as
in~\eqref{e:2.6}.  From~\eqref{e:2.17} it follows
\begin{equation}
  \tanh\psi=y\tan\varphi-x,\quad
  \sech\psi=\sqrt{1-\tanh^2\psi},\quad
  \cosh\psi=1/\sech\psi.
  \label{e:2.19}
\end{equation}

Theorem~\ref{t:2.3} is an analogon of Theorem~\ref{t:2.2} in this
hyperbolic case.

\begin{theorem}\label{t:2.3}
  If, in~\eqref{e:2.15}, $x\ne 1$, then $|\tanh\psi|<1$ and
  \begin{equation}
    \tilde{\sigma}_{11}'=\frac{\sec\varphi}{\sech\psi}\hat{t}_{11}^{}=\frac{\cosh\psi}{\cos\varphi}\hat{t}_{11}^{},\quad
    \tilde{\sigma}_{22}'=\frac{\sech\psi}{\sec\varphi}\hat{t}_{22}^{}=\frac{\cos\varphi}{\cosh\psi}\hat{t}_{22}^{}.
    \label{e:2.20}
  \end{equation}
\end{theorem}
\begin{proof}
  From~\eqref{e:2.18}--\eqref{e:2.19}, $0\le\tan\varphi<1$ and
  $|\tanh\psi|=|y\tan\varphi-x|$.  With $x$ fixed, $|\tanh\psi|$
  attains the maximum value for $y\tan\varphi$ being either the
  largest, i.e., when approaching unity from the left, or the smallest
  possible, i.e., zero.  In the first case, $|\tanh\psi|<|1-x|\le 1$.
  In the second, $|\tanh\psi|=|-x|=x$, so $x\ne 1\implies|\tanh\psi|<1$.

  Now~\eqref{e:2.20} is proven.  From~\eqref{e:2.15} and~\eqref{e:2.19}
  it follows
  \begin{align*}
    \tilde{\sigma}_{11}''&=y+x\tan\varphi+\tan\varphi\tanh\psi=y+\tan\varphi(x+\tanh\psi)=y+y\tan^2\varphi=y\sec^2\varphi,\\
    \tilde{\sigma}_{22}''&=1+(x-y\tan\varphi)\tanh\psi=1-\tanh^2\psi={\sech}^2\psi,
  \end{align*}
  what, after multiplying both equations by
  $\hat{t}_{22}^{}\cos\varphi\cosh\psi$, gives~\eqref{e:2.20}.
  \qed%
\end{proof}

Algorithm~\ref{a:2.5}, similarly to Algorithm~\ref{a:2.4}, computes
reproducibly and as stably as practicable the HSVD of a lower
triangular (including diagonal) $\widehat{T}$ when
$\widehat{J}\ne\pm I_2$.

\begin{algorithm}[hbtp]
  $x=\mathop{\mathtt{MAX}}(\hat{t}_{21}/\hat{t}_{22},0)$,\quad$y=\mathop{\mathtt{MAX}}(\hat{t}_{11}/\hat{t}_{22},0)$\tcp*[r]{\eqref{e:2.15} ($\mathtt{MAX}$ handles $\widehat{T}=\mathbf{0}$ case)}
  $\displaystyle\tan(2\varphi)=\min\left\{\mathop{\mathtt{MAX}}\left(\frac{(2\min\{x,y\})\max\{x,y\}}{\fma(y-x,y+x,1)},0\right),\alpha\right\}$\tcp*[r]{\eqref{e:2.18}}
  $\displaystyle\tan\varphi=\frac{\tan(2\varphi)}{1+\sqrt{\fma(\tan(2\varphi),\tan(2\varphi),1)}}$,\quad$\tanh\psi=\fma(y,\tan\varphi,-x)$\tcp*[r]{\eqref{e:2.6}~\&~\eqref{e:2.19}}
  \lIf(\tcp*[f]{$|\tanh\psi|$ is unsafe if it is too close to unity}){$|\tanh\psi|\ge\upsilon$}{\KwRet{$\bot$}}
  \eIf{{\rm an accurate enough (e.g., correctly rounded) $\mathop{\mathtt{rsqrt}}(x)=1/\sqrt{x}$ function is available}}{%
    $\cos\varphi=\mathop{\mathtt{rsqrt}}(\fma(\tan\varphi,\tan\varphi,1))$\tcp*[r]{\eqref{e:2.6}}
    $\cosh\psi=\mathop{\mathtt{rsqrt}}(\fma(-\tanh\psi,\tanh\psi,1))$\tcp*[r]{\eqref{e:2.19}}
    \Backscale{$s$, $\hat{t}_{11}^{}$, $\hat{t}_{22}^{}$, $\cosh\psi/\cos\varphi$, $\hat{\sigma}_{11}'$, $\hat{\sigma}_{22}'$}\tcp*[r]{\eqref{e:2.20} \&\ Algorithm~\ref{a:2.2}}
  }(\tcp*[f]{usually, $\mathop{\mathtt{rsqrt}}(x)$ is not available}){%
    $\sec\varphi=\sqrt{\fma(\tan\varphi,\tan\varphi,1)}$,\quad$\cos\varphi=1/\sec\varphi$\tcp*[r]{\eqref{e:2.6}}
    $\sech\psi=\sqrt{\fma(-\tanh\psi,\tanh\psi,1)}$,\quad$\cosh\psi=1/\sech\psi$\tcp*[r]{\eqref{e:2.19}}
    \Backscale{$s$, $\hat{t}_{11}^{}$, $\hat{t}_{22}^{}$, $\sec\varphi/\sech\psi$, $\hat{\sigma}_{11}'$, $\hat{\sigma}_{22}'$}\tcp*[r]{\eqref{e:2.20} \&\ Algorithm~\ref{a:2.2}}
  }
  \KwRet{$\top$}\tcp*[r]{the hyperbolic transformation is defined and safe}
  \caption{The HSVD of a lower triangular $\widehat{T}$ in \texttt{DOUBLE PRECISION}.}
  \label{a:2.5}
\end{algorithm}

Let $\varepsilon$, unless noted otherwise, be half the unit in the
last place (ulp) of $1$, or equivalently, $\varepsilon=1-{'1}$, where
$'1$ is the immediate floating-point predecessor of $1$ in the chosen
datatype; e.g., $\varepsilon=2^{-24}$ and $\varepsilon=2^{-53}$ in
single and double precision, respectively.  In Fortran, the latter is
equal to the result of
$\mathop{\mathtt{EPSILON}}(\mathtt{0D0})/\mathtt{2}$.

From~\eqref{e:2.13} and~\eqref{e:2.19} it follows
$\sqrt{\varepsilon}\le\sech\psi\le 1$, since the argument of the
secant's defining square root cannot be less than $\varepsilon$
whenever $|\tanh\psi|<1$ (what is implied by
$\tanh^2\psi\le|\tanh\psi|\le{'1}$).  Also, from~\eqref{e:2.5},
\eqref{e:2.12}, and~\eqref{e:2.18} it can be concluded that
$|\tan\varphi|<1$, so $1\le\sec\varphi\le\sqrt{2}$ and
$1\le\sec\varphi/\sech\psi\le\sqrt{2/\varepsilon}$.  In the
trigonometric case, $1\le\sec\psi\le\sqrt{3}$ and
$\sec\varphi\le\sec\psi$, due to~\eqref{e:2.7} and
Theorem~\ref{t:2.1}, so
$1\le\sec\psi/\sec\varphi\le\sqrt{3}<\sqrt{2/\varepsilon}$.  Hence
$f\le\sqrt{2/\varepsilon}$ in the uses of Algorithm~\ref{a:2.2}.
\subsection{The HSVD of $\widehat{G}$}\label{ss:2.3}
The HSVD of $\widehat{G}$ as
$\widehat{U}^{\ast}\widehat{G}\widehat{V}=\widehat{\Sigma}$ is
obtained in four phases:
\begin{compactenum}
\item computing the scaling parameter $s$ and
  $\widehat{G}_0=2^s\widehat{G}$ as in subsection~\ref{sss:2.1.1},
\item factoring $\widehat{G}_0^{}$ as
  $\check{U}^{\ast}\widehat{G}_0^{}\check{V}=\widehat{T}$ by
  Algorithm~\ref{a:2.1},
\item computing the HSVD of $\widehat{T}$,
  $\widetilde{U}^{\ast}\widehat{T}\widetilde{V}=\widetilde{\Sigma}$
  and $\widehat{\Sigma}=2^{-s}\widetilde{\Sigma}$, as in
  subsection~\ref{ss:2.2}, and
\item assembling
  $\widehat{U}^{\ast}=\widetilde{U}^{\ast}\check{U}^{\ast}$ and
  $\widehat{V}=\check{V}\widetilde{V}$ (and, if required,
  $\widehat{V}^{-1}=\widehat{J}\widehat{V}^{\ast}\widehat{J}$).
\end{compactenum}

The third phase can fail if the HSVD is not defined, or is unsafe, as
indicated by the return values of Algorithm~\ref{a:2.4} and
Algorithm~\ref{a:2.5}.  Safety (or a lack thereof), parametrized by
$\upsilon$ in those algorithms, is a user's notion of how close
$|\tanh\psi|$ can get to unity from below to still define a hyperbolic
rotation that is well-enough conditioned to be applied to the pivot
columns of the iteration matrix.  For the $2\times 2$ HSVD in
isolation, $\upsilon=1$, but for the $n_0\times n_0$ HSVD it might
sometimes be necessary to set it a bit lower if the HSVD
process otherwise fails to converge, e.g., to $\upsilon=0.8$, as
in~\cite{Veselic-93}, or even lower.  There is no prescription for
choosing an adequate $\upsilon$ in advance; if an execution takes far
more (multi-)steps than expected (see section~\ref{s:6} for some
estimates in terms of cycles), it should be aborted and restarted with
a lower $\upsilon$.

If $\upsilon<1$, $|\tanh\psi|<1$, and the return value from either
algorithm is $\bot$, there are two possibilities.  First, $\tanh\psi$
can be set to $\sign(\tanh\psi)\upsilon$, and the hyperbolic
transformation can be computed accordingly~\cite{Slapnicar-92}.
However, the off-diagonal elements of the transformed pivot matrix
cannot be considered zeros anymore, and they have to be formed in the
iteration matrix and included in the weight computations (see
subsection~\ref{ss:3.2}).  An inferior but simpler solution declares
that the $2\times 2$ HSVD is not defined, as if $|\tanh\psi|\ge 1$,
and the pivot pair in question will no longer be a transformation
candidate in the current multi-step, in the hope that it will again
become one, with a better conditioned hyperbolic rotation, after its
pivot row(s) and column(s) have been sufficiently transformed.  In the
implementation the latter option has been chosen, but the former might
lead (not tested) to a faster convergence when it has to be employed.
\section{Row and column transformations}\label{s:3}
If the HSVD of $\widehat{G}_{k-1}$ is not defined, the algorithm
stops.  Else, having computed $\widehat{U}_k^{\ast}$,
$\widehat{\Sigma}_k$, and $\widehat{V}_k$ for a transformation
candidate with the pivot indices $p_k$ and $q_k$, the $p_k$th and the
$q_k$th row of $G_{k-1}$ are transformed by multiplying them from the
left by $\widehat{U}_k^{\ast}$,
\begin{displaymath}
  \begin{bmatrix}
    G_{k-1}'(p_k^{},:)\\
    G_{k-1}'(q_k^{},:)
  \end{bmatrix}=
  \widehat{U}_k^{\ast}
  \begin{bmatrix}
    G_{k-1}^{}(p_k^{},:)\\
    G_{k-1}^{}(q_k^{},:)
  \end{bmatrix}.
\end{displaymath}
Then, $G_k^{}$ is obtained from $G_{k-1}'$ after transforming the
$p_k^{}$th and the $q_k^{}$th column of $G_{k-1}'$ by multiplying them
from the right by $V_k^{}$,
\begin{displaymath}
  \begin{bmatrix}
    G_k^{}(:,p_k^{}) & G_k^{}(;,q_k^{})
  \end{bmatrix}=
  \begin{bmatrix}
    G_{k-1}'(:,p_k^{}) & G_{k-1}'(:,q_k^{})
  \end{bmatrix}
  V_k,
\end{displaymath}
and setting $G_k(p_k,p_k)$ to the first diagonal element of
$\widehat{\Sigma}_k$, $G_k(q_k,q_k)$ to the second one (in both cases
reusing the possibly more accurate hyperbolic singular values from the
HSVD of $\widehat{G}_{k-1}$ then those computed by the row and the
column transformations of $G_k$), while explicitly zeroing out
$G_k(p_k,q_k)$ and $G_k(q_k,p_k)$.

If the left and the right (hyperbolic) singular vectors are desired,
in a similar way as above the current approximations of $U^{\ast}$ and
$V$ are updated by $\widehat{U}_k^{\ast}$ and $\widehat{V}_k^{}$,
respectively.
\subsection{Effects of a hyperbolic transformation}\label{ss:3.1}
If $\widehat{V}_k$ is unitary, the square of the off-diagonal
Frobenius norm of the transformed $G_k$ is reduced by
$|G_{k-1}(q_k,p_k)|^2+|G_{k-1}(p_k,q_k)|^2\ge 0$.  Else, if
$\widehat{V}_k$ is $\widehat{J}_k$-unitary, the following
Lemma~\ref{l:3.1} sets the bounds to the relative change of the square
of the Frobenius norm of the $p_k$th and the $q_k$th columns
multiplied from the right by a hyperbolic rotation.  Note that a
complex $\widehat{V}_k$ is of the form
\begin{displaymath}
  \widehat{V}_k=
  \begin{bmatrix}
    1 & 0\\
    0 & e^{-\mathrm{i}\beta}
  \end{bmatrix}
  \begin{bmatrix}
    \cosh\psi & \sinh\psi\\
    \sinh\psi & \cosh\psi
  \end{bmatrix}=
  \begin{bmatrix}
    \hphantom{e^{-\mathrm{i}\beta}}\cosh\psi & \hphantom{e^{-\mathrm{i}\beta}}\sinh\psi\\
    e^{-\mathrm{i}\beta}\sinh\psi & e^{-\mathrm{i}\beta}\cosh\psi
  \end{bmatrix},
\end{displaymath}
which can alternatively be expressed as
\begin{displaymath}
  \begin{bmatrix}
    \hphantom{e^{-\mathrm{i}\beta}}\cosh\psi & \hphantom{e^{-\mathrm{i}\beta}}\sinh\psi\\
    e^{-\mathrm{i}\beta}\sinh\psi & e^{-\mathrm{i}\beta}\cosh\psi
  \end{bmatrix}=
  \begin{bmatrix}
    \hphantom{e^{-\mathrm{i}\beta}}\cosh\psi & e^{\mathrm{i}\beta}\sinh\psi\\
    e^{-\mathrm{i}\beta}\sinh\psi & \hphantom{e^{\mathrm{i}\beta}}\cosh\psi
  \end{bmatrix}
  \begin{bmatrix}
    1 & 0\\
    0 & e^{-\mathrm{i}\beta}
  \end{bmatrix},
\end{displaymath}
where the rightmost matrix (call it $B$) is unitary and does not
change the Frobenius norm of any matrix that multiplies it from the
left.  Lemma~\ref{l:3.1} is thus stated in full generality, for
$\widehat{V}_k B^{\ast}$, since
$\|X\widehat{V}_k B^{\ast}\|_F=\|X\widehat{V}_k B^{\ast}B\|_F=\|X\widehat{V}_k\|_F$
for all conformant $X$.

\begin{lemma}\label{l:3.1}
  If $\mathbf{x}$ and $\mathbf{y}$ are complex vectors of length $n$,
  such that
  $\left\|\begin{bmatrix}\mathbf{x}&\mathbf{y}\end{bmatrix}\right\|_F>0$,
  and
  \begin{displaymath}
    \begin{bmatrix}
      \mathbf{x}' & \mathbf{y}'
    \end{bmatrix}=
    \begin{bmatrix}
      \mathbf{x} & \mathbf{y}
    \end{bmatrix}
    \begin{bmatrix}
      \hphantom{e^{-\mathrm{i}\beta}}\cosh\psi & e^{\mathrm{i}\beta}\sinh\psi\\
      e^{-\mathrm{i}\beta}\sinh\psi & \hphantom{e^{\mathrm{i}\beta}}\cosh\psi
    \end{bmatrix},
  \end{displaymath}
  then
  \begin{displaymath}
    \cosh(2\psi)-|\sinh(2\psi)|\le
    \frac{\left\|\begin{bmatrix}\mathbf{x}'&\mathbf{y}'\end{bmatrix}\right\|_F^2}{\left\|\begin{bmatrix}\mathbf{x}&\mathbf{y}\end{bmatrix}\right\|_F^2}\le
    \cosh(2\psi)+|\sinh(2\psi)|.
  \end{displaymath}  
\end{lemma}

The following Lemma~\ref{l:3.2} refines the bounds stated in
Lemma~\ref{l:3.1}.  Both Lemmas, proved in Appendix~\ref{s:A}, are
used in the proof of Theorem~\ref{t:3.1} in subsection~\ref{ss:3.2}.

\begin{lemma}\label{l:3.2}
  In the lower or the upper bound established in Lemma~\ref{l:3.1} the
  equality is attainable if and only if
  $\mathbf{y}=\pm e^{\mathrm{i}\beta}\mathbf{x}$ or $\psi=0$.  The
  lower bound is always positive but at most unity, and the upper
  bound is at least unity.
\end{lemma}

Another observation is that the norm of the transformed columns
depends both on the norm of the original columns, as well as on the
hyperbolic transformation applied.  Therefore, $\psi$ of a relatively
large magnitude does not by itself pose a problem if the original
columns have a modest norm.  And contrary, even $\psi$ of a relatively
small magnitude can---and in practice, will---cause the columns'
elements of a huge magnitude, should such exist, to overflow in the
finite machine arithmetic.
\subsection{Weight of a transformation candidate}\label{ss:3.2}
Let $\off(A)$ be the square of the off-diagonal Frobenius norm of
$A\in\mathbb{F}^{n\times n}$, i.e.,
\begin{displaymath}
  \off(A)=\sum_{j=1}^n\sum_{\substack{i=1\\i\ne j}}^n|a_{ij}^{}|^2=\|A-\diag(a_{11}^{},\ldots,a_{nn}^{})\|_F^2.
\end{displaymath}
The following Theorem~\ref{t:3.1} deals with the amount of change
$\off(G_{k-1}^{})-\off(G_k^{})$.

\begin{theorem}\label{t:3.1}
  For $k$ such that $1\le k\le N$ and
  $w_k^{}=\off(G_{k-1}^{})-\off(G_k^{})$ it holds
  \begin{equation}
    w_k^{}=|G_{k-1}^{}(q_k^{},p_k^{})|^2+|G_{k-1}^{}(p_k^{},q_k^{})|^2+h_k^{},
    \label{e:3.2}
  \end{equation}
  where $h_k=0$ and $w_k$ is non-negative if $V_k$ is unitary.
  Otherwise, for $h_k$ holds
  \begin{equation}
    h_k^{}=\smashoperator{\sum_{\substack{i=1\\i\notin\{p_k^{},q_k^{}\}}}^n}
    \left((|G_{k-1}^{}(i,p_k^{})|^2-|G_k^{}(i,p_k^{})|^2)+(|G_{k-1}^{}(i,q_k^{})|^2-|G_k^{}(i,q_k^{})|^2)\right),
    \label{e:3.3}
  \end{equation}
  and $w_k$ can be negative, positive, or zero.
\end{theorem}
\begin{proof}
  When $V_k$ is unitary, the statement of the Theorem~\ref{t:3.1} is a
  well-known property of the Kogbetliantz algorithm, and a consequence
  of $U_k^{\ast}$ also being unitary, as well as the Frobenius norm
  being unitary invariant.

  Else, if $V_k$ is not unitary, then observe that the elements of
  $G_k$ at the pivot positions (i.e., having their indices taken from
  the set $\{p_k,q_k\}$) do not contribute to $\off(G_k^{})$, since
  the off-diagonal elements at those positions are zero.  The change
  from $\off(G_{k-1}^{})$ to $\off(G_k^{})$ is therefore the sum of
  squares of the magnitudes of those elements, plus any change ($h_k$)
  happening outside the pivot positions, as in~\eqref{e:3.2}.

  The left transformation is unitary, and therefore the only two rows
  affected, $p_k$th and $q_k$th, shortened to have the pivot elements
  removed, keep their joint Frobenius norm unchanged from $G_{k-1}^{}$
  to $G_{k-1}'$.  Since the right transformation affects only the
  $p_k$th and the $q_k$th column, either of which intersect the
  $p_k$th and the $q_k$th row in the pivot positions only, there is no
  further change from $\off(G_{k-1}')$ to $\off(G_k^{})$ when the
  off-diagonal norm is restricted to the shortened and transformed
  rows, so a contribution to $h_k$ from the left transformation is
  zero.

  Therefore, only the right transformation is responsible for the
  value of $h_k$, which can be bounded by Lemma~\ref{l:3.1}, applied
  to the computed hyperbolic transformation $\widehat{V}_k$ and the
  $p_k$th and the $q_k$th column with the pivot elements removed from
  them.  The square of the joint Frobenius norm of the shortened
  columns might either fall or rise after the transformation, due to
  Lemma~\ref{l:3.2}.  If it falls, $h_k$ and thus $w_k$ is positive.
  If it rises, depending on the hyperbolic angle $\psi$ and on the
  off-diagonal elements, $h_k$ can become negative and so large in
  magnitude to push $w_k$ down to zero or below.

  For example, let $0<\epsilon\ll 1$ and observe that in the real case
  \begin{displaymath}
    \widehat{G}_{k-1}=
    \begin{bmatrix}
      1 & \epsilon-1\\
      0 & 0
    \end{bmatrix}
    \implies
    \widehat{V}_k=\cosh\psi
    \begin{bmatrix}
      1 & \epsilon-1\\
      1-\epsilon & -1
    \end{bmatrix},\quad
    \cosh\psi=\frac{1}{\sqrt{\epsilon(2-\epsilon)}},
  \end{displaymath}
  (see subsection~\ref{sss:2.2.2}).  Let $a$ and $b$ lie in $p_k$th
  and the $q_k$th column, respectively, in a row
  $\ell\notin\{p_k,q_k\}$, and let the other non-pivot elements be
  zero.  Then $a$ and $b$ are transformed to $\tilde{a}$ and
  $\tilde{b}$, respectively, where
  \begin{displaymath}
    \begin{bmatrix}
      \tilde{a} & \tilde{b}
    \end{bmatrix}=
    \begin{bmatrix}
      a & b
    \end{bmatrix}
    \widehat{V}_k=
    \begin{bmatrix}
      a+(1-\epsilon)b & a(\epsilon-1)-b
    \end{bmatrix}/\sqrt{\epsilon(2-\epsilon)}.
  \end{displaymath}
  If $a=b=1$, then
  $\tilde{a}=-\tilde{b}=\sqrt{2-\epsilon}/\sqrt{\epsilon}$, so
  \begin{displaymath}
    a^2-\tilde{a}^2+b^2-\tilde{b}^2=4(\epsilon-1)/\epsilon\approx-4/\epsilon,
  \end{displaymath}
  what is by magnitude far greater than $0^2+|\epsilon-1|^2\approx 1$,
  thus $w_k<0$.  Oppositely,
  \begin{displaymath}
    \widehat{G}_{k-1}=
    \begin{bmatrix}
      1 & 1-\epsilon\\
      0 & 0
    \end{bmatrix}
    \implies
    \widehat{V}_k=\cosh\psi
    \begin{bmatrix}
      1 & \epsilon-1\\
      \epsilon-1 & 1
    \end{bmatrix},\quad
    \cosh\psi=\frac{1}{\sqrt{\epsilon(2-\epsilon)}},
  \end{displaymath}
  while $\tilde{a}$ and $\tilde{b}$ are
  \begin{displaymath}
    \tilde{a}=(a+(\epsilon-1)b)/\sqrt{\epsilon(2-\epsilon)},\quad
    \tilde{b}=(a(\epsilon-1)+b)/\sqrt{\epsilon(2-\epsilon)}.
  \end{displaymath}
  If $a=b=1$, then
  $\tilde{a}=\tilde{b}=\sqrt{\epsilon}/\sqrt{2-\epsilon}$, so
  \begin{displaymath}
    a^2-\tilde{a}^2+b^2-\tilde{b}^2=4(1-\epsilon)/(2-\epsilon)\approx 2,
  \end{displaymath}
  thus $w_k>0$.  Finally, let $a>0$, $b=0$, and assume
  $\check{V}_k=I_2$ for simplicity.  Then
  \begin{displaymath}
    \begin{bmatrix}
      \tilde{a} & \tilde{b}
    \end{bmatrix}=
    \begin{bmatrix}
      a & 0
    \end{bmatrix}
    \begin{bmatrix}
      \cosh\psi & \sinh\psi\\
      \sinh\psi & \cosh\psi
    \end{bmatrix}=
    \begin{bmatrix}
      a\cosh\psi & a\sinh\psi
    \end{bmatrix},
  \end{displaymath}
  what, together with $1=\cosh^2\psi-\sinh^2\psi$, gives
  \begin{displaymath}
    a^2-\tilde{a}^2+b^2-\tilde{b}^2=a^2(1-\cosh^2\psi-\sinh^2\psi)=-2a^2\sinh^2\psi,
  \end{displaymath}
  what is equal to $-2a^2$ for $|\sinh\psi|=1$.  Setting
  \begin{displaymath}
    \widehat{G}_{k-1}=
    \begin{bmatrix}
      2a & a\sqrt{2}\\
      0 & 0
    \end{bmatrix}
    \implies
    \widehat{V}_k=
    \begin{bmatrix}
      \sqrt{2} & -1\\
      -1 & \sqrt{2}
    \end{bmatrix},\quad
    \sinh\psi=-1.
  \end{displaymath}
  The sum of squares of the off-diagonal elements of
  $\widehat{G}_{k-1}$ is $2a^2$ and thus $w_k=0$.
  \qed%
\end{proof}

A sequence of matrices $(G_k)_{k\ge 0}$ converges to a diagonal form
if and only if $(\off(G_k^{}))_{k\ge 0}$ tends to zero.  The sequence
$(\off(G_k^{}))_{k\ge 0}$ does not have to be monotonically
decreasing, i.e., $w_k$ can be negative for some $k$ in the HSVD case,
unlike in the ordinary Kogbetliantz algorithm.  That significantly
complicates any reasoning about convergence in theory, and attainment
of (a satisfactory rate of) convergence in practice.  To aid the
latter, the pivot weights $w_k$ from~\eqref{e:3.2} should be kept as
high as possible, by a careful choice of the pivot submatrix among all
admissible transformation candidates in each step.  In the SVD
computation, choosing a candidate with the maximal weight guarantees
convergence~\cite{Oksa-et-al-19}, and is trivially accomplished since
the weights are directly computable, due to $h_k=0$.  In the
hyperbolic case, $h_k$ from~\eqref{e:3.3} cannot be known in advance,
without performing the right (column) transformation, and it cannot be
estimated (roughly, due to Lemma~\ref{l:3.2}) by means of
Lemma~\ref{l:3.1} without computing the associated $\widehat{V}_k$ and
occasionally recomputing the column norms, what is of the same linear
complexity as transforming those columns.

The computation of $h_k$ is therefore preferable to an estimation.
In each step, it has to be performed for \emph{all} admissible
transformation candidates, and non-trivially for all indices $p$ and
$q$ such that $p<q$, $j_{pp}\ne j_{qq}$, and the associated
$\widehat{V}_k\ne I_2$ (if $\widehat{V}_k$ is not defined, let
$h_k=-\infty$ instead of halting), i.e., at most $n_0^2/4$ times, if
$J_0$ has the same number of positive and negative signs.  The left
transformation by $\widehat{U}_k^{\ast}$ is not needed here, so only
the right one by $\widehat{V}_k$ has to have a virtual variant, that
does not change the elements of $G_{k-1}$, but for each row
$i\notin\{p,q\}$ computes what would $G_k(i,p)$ and $G_k(i,q)$ be from
$G_{k-1}(i,p)$, $G_{k-1}(i,q)$, and $\widehat{V}_k$, and updates $h_k$
using~\eqref{e:3.3}.  That can be done accurately by applying twice
for each $i\notin\{p,q\}$ the accumulation rule
\begin{displaymath}
  \rho\mathop{:=}\rho+(a^2-\tilde{a}^2)=\fma(a-\tilde{a},a+\tilde{a},\rho),
\end{displaymath}
once for $a=|G_{k-1}(i,p)|$ and $\tilde{a}=|G_k(i,p)|$, and again for
$a=|G_{k-1}(i,q)|$ and $\tilde{a}=|G_k(i,q)|$.  If $\rho$ is
initialized to the sum of squares of the magnitudes of the
off-diagonal elements of $\widehat{G}_{k-1}$ instead of to zero, then
the final $\rho$ is $w_k$ instead of $h_k$.

All virtual transformations in a step are mutually independent, so
they should be performed in parallel.  With enough memory available
($O(n_0^3)$ in parallel), the computed elements of
$\widehat{U}_k^{\ast}$, $\widehat{V}_k$, and $G_k$ could be stored,
separately for each virtual transformation, and reused if the
corresponding candidate has been selected as a pivot.  Sequentially,
only the data ($O(n_0)$ values) of a candidate with the maximal weight
should be stored.  In the tested prototype of the algorithm the
virtual transformations are performed in parallel, but all data
generated, save the weights, are discarded.

It now emerges that each step requires at least $O(n_0^2)$, and at
most $O(n_0^3)$ operations just for computing all $w_k$.  The
complexity of the whole HSVD algorithm is therefore quartic in $n_0$
in the general case, far worse than the usual cubic algorithms for the
ordinary SVD.  It is legitimate to ask what impedes development of a
Kogbetliantz-type HSVD algorithm with a cubic complexity.  For that,
the pivot strategy should either ignore the weights and select the
pivots in a prescribed order (as, e.g., the row-cyclic or the
column-cyclic serial pivot strategies do), or compute all weights in a
step with no more than a quadratic number of operations (e.g., by
ignoring $h_k$ in~\eqref{e:3.2} and calculating the rest of $w_k$).
Either approach works well sometimes, but the former failed in the
numerical experiments when $n_0$ went up to $2000$, and the latter
even with $n_0$ around $100$, both with a catastrophic increase
(overflow) of the off-diagonal norm.  It remains an open question
whether employing in both cases a much wider floating-point datatype
(precision-wise as well as exponent-wise) could save the computation
and eventually lead to convergence, but that is of more interest to
theory than practice.  The pivot strategy motivated here and described
in detail in section~\ref{s:4}, however slow, is designed to keep the
growth of the off-diagonal norm minimal (at least with a single pivot)
when it is unavoidable, and can be used for the matrix orders of up to
a few thousands with enough parallelism at hand.  A faster but at
least equally safe pivot strategy would make the $J$-Kogbetliantz
algorithm competitive in performance with the pointwise one-sided
hyperbolic Jacobi method.

It is remarkable that the extensive numerical tests conducted with
many variants of the parallel blocked one-sided hyperbolic Jacobi
method, all of them with some prescribed cyclic pivot strategy, either
on CPU~\cite{Singer-et-al-12a,Singer-et-al-12b} or on
GPU~\cite{Novakovic-15,Novakovic-Singer-11}, have never shown an
indication of a dangerous off-diagonal norm growth, even though the
hyperbolic angle of a $2\times 2$ transformation might be of a large
magnitude there as well, at least in theory.  It may be interesting to
look further into why that did not happen (and probably is hard to
make happen) in practice, unlike with the Kogbetliantz-type algorithm,
which requires such a complex pivot strategy to reduce the norm
growth.
\subsection{Floating-point considerations}\label{ss:3.3}
The sum of squares from~\eqref{e:3.2} can overflow, as well as each of
the squares, leading to $w_k=\infty$.  A dynamic rescaling of the
whole $G_{k-1}$ by an appropriate power of two could mitigate that
issue, but with a risk that the smallest values by magnitude become
subnormal and lose precision.   However, if the weights are computed
in a wide enough floating-point datatype (e.g., using the Intel's
80-bit extended), no overflow can occur.  Similarly, one or both
squares can underflow (a fact to be relied upon in section~\ref{s:5})
to a point of becoming zero(s).  Then, if $h_k=0$, the only way of
avoiding $w_k=0$, without computing in a wider datatype, is scaling
$G_{k-1}$ upwards, thus risking overflow of the largest elements by
magnitude.  As a partial remedy, the augmented weights from
subsection~\ref{ss:4.1} are always distinct and well-ordered, so a
deterministic pivot selection is possible even with some (or all)
weights being $\infty$ or underflowing.

There is no rule of thumb how to properly prescale $G_0$, so that such
issues, as well as the potential overflows due to the hyperbolic
transformations, do not needlessly occur.  Monitoring the computed
weights can indicate should the latter problems be immediately avoided
by downscaling $G_{k-1}$.  In the tested prototype of the algorithm the
dynamic scaling of the whole matrix, unlike the scaling from
subsection~\ref{sss:2.1.1}, has not been implemented, but should
otherwise be if robustness is paramount.
\section{Dynamic pivot selection based on weights}\label{s:4}
A dynamic pivot strategy (DPS in short) based on block weights was
introduced in~\cite{Becka-Oksa-Vajtersic-02} for the two-sided
block-Jacobi SVD algorithm (as the block-Kogbetliantz algorithm is
also called), while the global and the asymptotic quadratic
convergence of such a coupling was proven in~\cite{Oksa-et-al-19} for
the serial (a single block pair per step) and the parallel (multiple
block pairs per step) annihilation.  As the pointwise Jacobi
algorithms are but a special case of the block ones, when the blocks
(matrices) contain only one, scalar element, all properties of the
dynamic pivoting hold in that context as well.

However, a DPS used in the pointwise Kogbetliantz-type HSVD algorithm
differs in several aspects from the one for the SVD.  A weight, i.e.,
the amount of the off-diagonal norm reduction, in the latter is finite
(up to a possible floating-point overflow) and non-negative, while in
the former it can be of arbitrary sign and infinite.  Yet, the goal in
both cases is the same: to reduce the off-diagonal norm in each step
as much as possible.  In the latter the off-diagonal norm growth is
impossible, while in the former it is sometimes necessary, but is
still kept as low as practicable.

Another important difference is in handling a situation when some or
all weights are the same.  In the former, a concept of augmented
weight is introduced, as follows.
\subsection{DPS in the sequential case}\label{ss:4.1}
\begin{definition}\label{d:4.1}
  Let the weight of a $2\times 2$ submatrix of $G_{k-1}$ at the
  intersection of the $p$th and the $q$th row with the $p$th and the
  $q$th column be computed according to~\eqref{e:3.2} if that
  submatrix is a transformation candidate.  Else, if the submatrix
  does not need to be transformed, or cannot be transformed due to at
  least one its elements being non-finite, define its weight as a
  quiet $\mathtt{NaN}$.  Let a triple $w_{pq}^{[k]}=(w,p,q)$ be called
  an \emph{augmented weight}, where $w$ is the weight of the submatrix
  induced by $(p,q)$.  Also, let
  $\mathbf{w}_k^{}=\{w_{pq}^{[k]}\mid 1\le p<q\le n_0\}$ be the set of
  all augmented weights in the $k$th step.
\end{definition}

For any given $k$, a total order $\preceq$ can be defined on the
augmented weights that makes all of them distinct, even though the
weights themselves may be equal.

\begin{definition}\label{d:4.2}
  Let, for some $k$, $a$ and $b$ be two augmented weights in
  $\mathbf{w}_k$, and let them be considered equal, denoted as $a=b$,
  if and only if their corresponding components are equal, i.e.,
  $a.w=b.w$, $a.p=b.p$, and $a.q=b.q$.  Contrary to the usual
  definition of $\mathtt{NaN}$, in this context let
  $\mathtt{NaN}=\mathtt{NaN}$ and $\mathtt{NaN}<c$ for any other $c$.
  Let $\preceq$ be the union of the relations $\prec$ and $=$, where
  $a\prec b$ if and only if
  \begin{compactenum}
  \item $a.w>b.w$, or
  \item $a.w=b.w$ and $a.q-a.p>b.q-b.p$, or
  \item $a.w=b.w$, $a.q-a.p=b.q-b.p$, and $a.q>b.q$.
  \end{compactenum}
\end{definition}

\begin{proposition}\label{p:4.1}
  The relation $\preceq$ from Definition~\ref{d:4.2} makes
  $\mathbf{w}_k$ well ordered; specifically, every non-empty subset of
  $\mathbf{w}_k$, including $\mathbf{w}_k$, has a unique
  $\preceq$-smallest element.
\end{proposition}

\begin{proof}
  It is easy to verify that $\preceq$ is a total order on
  $\mathbf{w}_k$.  Since $\mathbf{w}_k$ is finite, it is well
  ordered by $\preceq$.  If all weights in $S$,
  $\emptyset\ne S\subseteq\mathbf{w}_k$, are different, the smallest
  element is the one with the largest weight (due to condition~1 from
  Definition~\ref{d:4.2}).  The quantities $a.q-a.p$ and $b.q-b.p$
  indicate a band, i.e., a sub/super-diagonal of $G_{k-1}$ at which
  $(q,p)$ and $(p,q)$ lie, respectively, with the main diagonal being
  band $0$.  If several elements of $S$ have the same maximal weight,
  the smallest element is the one among them in the farthest band
  (condition~2).  If more than one such element exists, the smallest
  is the one lying lowest, i.e., with the largest column index
  (condition~3).
  \qed%
\end{proof}

The following Corollary~\ref{c:4.1} is a direct consequence of
Proposition~\ref{p:4.1} and defines the DPS in the sequential case,
i.e., when only one pivot is transformed in each step.

\begin{corollary}\label{c:4.1}
  Let
  $\widetilde{\mathbf{w}}_k^{}=\mathbf{w}_k^{}\setminus\{a\mid a.w=\mathtt{NaN}\}$
  be a set of augmented weights such that the weights themselves are
  not $\mathtt{NaN}$.  If $\widetilde{\mathbf{w}}_k=\emptyset$, no
  transformations are possible and the algorithm stops with $N=k$.
  Else, let $\hat{a}$ be the $\preceq$-smallest element of
  $\widetilde{\mathbf{w}}_k$.  If $\hat{a}.w=-\infty$, no
  transformation is valid and the algorithm halts with an error.
  Else, $(\hat{a}.p,\hat{a}.q)$ are the indices of a single pivot to
  be chosen in the $k$th step.
\end{corollary}

Finding the smallest element of $\widetilde{\mathbf{w}}_k$ is linear
in $c=|\widetilde{\mathbf{w}}_k|$, i.e., at most quadratic in $n_0$,
if a na\"{\i}ve method is used.  However, any $t$ disjoint subsets of
$\widetilde{\mathbf{w}}_k$, each of them of size at most
$\lceil c/t\rceil$ and at least one less that, can be linearly
searched for their smallest elements, all of them in parallel.  The
smallest elements thus found can in turn be $\preceq$-reduced in
parallel with $\lceil\log_2 t\rceil$ complexity to get the smallest
element overall.

Furthermore, observe that only the weights in the pivot rows and
columns change after a step.  Then, in the next step, the weights in
the changed positions have to be recomputed and compared with the
unchanged weights in the remaining part of the matrix, for which the
$\preceq$-smallest element can already be found in the previous step.
Therefore, in each step two elements of $\widetilde{\mathbf{w}}_k$
have to be found: the $\preceq$-smallest one $a$, and its closest
$\preceq$-successor $b$ such that
$\{a.p,a.q\}\cap\{b.p,b.q\}=\emptyset$.  Finding such $a$ and $b$
would be easiest if $\widetilde{\mathbf{w}}_k$ would have already been
sorted $\preceq$-ascendingly.  But if such $a$ and $b$ are found, they
define two pivots that can both be transformed in parallel, i.e., in a
multi-step of length two.  Repeating the observation of this
paragraph, both a sketch of a method and an argument for the parallel
DPS emerges, where a sequence of pivots, all with their indices
disjoint, is incrementally built to be transformed in a multi-step.
The case of a single pivot per step is here abandoned in favor of the
parallel, multi-step case, albeit it can be noticed that some pivots
in such a multi-step can lead to the off-diagonal norm growth when the
$\preceq$-smallest one does not.  The sequential case is thus locally
(i.e., in each step, but not necessarily globally) optimal with
respect to the change of the off-diagonal norm, but the parallel one
may not be.
\subsection{DPS in the multi-step case}\label{ss:4.2}
For a multi-step $\mathbf{k}$, let the augmented weights
$w_{pq}^{[\mathbf{k}]}$ and the set $\mathbf{w}_{\mathbf{k}}$ of them
be defined as in Definition~\ref{d:4.1}, with the smallest
$k\in\mathbf{k}$.  Definition~\ref{d:4.2}, Proposition~\ref{p:4.1},
and Corollary~\ref{c:4.1} are then modified accordingly.  Also, let
$\widehat{\mathbf{w}}_{\mathbf{k}}$ be a $\preceq$-ascendingly sorted
array of the elements of
$\widetilde{\mathbf{w}}_{\mathbf{k}}\setminus\{a\mid a.w=-\infty\}$.
An option to get $\widehat{\mathbf{w}}_{\mathbf{k}}$ from
$\widetilde{\mathbf{w}}_{\mathbf{k}}$ is the parallel merge sort.  In
the prototype implementation, the Baudet--Stevenson odd-even sort with
merge-splitting of the subarrays~\cite{Baudet-Stevenson-78} is used,
since it is simple and keeps $\lceil t/2\rceil$ tasks active at any
given time, even though its worst-case complexity is quadratic.  Both
choices require a work array of $c$ augmented weights, but that
scratch space can be reused elsewhere.  For $t=1$, a sequential merge
(or quick) sort is applicable.

\begin{definition}\label{d:4.3}
  Let $\mathbf{S}_{\mathbf{k}}$ be the set of all $\preceq$-ascending
  sequences of length at most $|\mathbf{k}|$ of the augmented weights
  from $\widehat{\mathbf{w}}_{\mathbf{k}}$ with non-intersecting
  indices, i.e., of all (not necessarily contiguous) subarrays of
  $\widehat{\mathbf{w}}_{\mathbf{k}}$ of length at most
  $|\mathbf{k}|$, such that for any two elements $a$ and $b$ from a
  subarray holds $\{a.p,a.q\}\cap\{b.p,b.q\}=\emptyset$.  Let
  $S_{\mathbf{k}}\in\mathbf{S}_{\mathbf{k}}$ be arbitrary, $m\ge 1$ be
  the length of $S_{\mathbf{k}}$, and define the following functions
  of $S_{\mathbf{k}}$,
  \begin{displaymath}
    w(S_{\mathbf{k}}^{})=\sum_{\ell=1}^mS_{\mathbf{k}}^{}(\ell).w,\quad
    \mathbf{o}(S_{\mathbf{k}}^{})=(l_{\ell}^{}\mid S_{\mathbf{k}}^{}(\ell)=\widehat{\mathbf{w}}_{\mathbf{k}}^{}(l_{\ell}^{}))_{\ell=1}^m,
  \end{displaymath}
  as its weight and as a sequence of indices that its elements have in
  $\widehat{\mathbf{w}}_{\mathbf{k}}$, respectively.
\end{definition}

It suffices to restrict Definition~\ref{d:4.3} to the sequences of
length $m>0$ only, since $|\widehat{\mathbf{w}}_{\mathbf{k}}|=0$
implies that no valid transformations are possible, and the execution
halts.

\begin{definition}\label{d:4.4}
  Let $|\mathbf{k}|$, $\widehat{\mathbf{w}}_{\mathbf{k}}$, and $\ell$
  such that $1\le\ell\le|\widehat{\mathbf{w}}_{\mathbf{k}}|$ be
  given.  Then, a \emph{parallel ordering}
  $O_{\mathbf{k}}^{\ell}\in\mathbf{S}_{\mathbf{k}}^{}$ is the sequence
  of maximal length, but not longer than $|\mathbf{k}|$, such that
  $O_{\mathbf{k}}^{\ell}(1)=\widehat{\mathbf{w}}_{\mathbf{k}}^{}(\tau_1)$,
  with $\tau_1=\ell$, and
  $O_{\mathbf{k}}^{\ell}(l)=\widehat{\mathbf{w}}_{\mathbf{k}}^{}(\tau_l)$
  for $l>1$, where $\tau_l$ is the smallest index of an element of
  $\widehat{\mathbf{w}}_{\mathbf{k}}$ such that
  \begin{equation}
    \{\widehat{\mathbf{w}}_{\mathbf{k}}(\tau_l).p,\widehat{\mathbf{w}}_{\mathbf{k}}(\tau_l).q\}\cap\{\widehat{\mathbf{w}}_{\mathbf{k}}(\tau_i).p,\widehat{\mathbf{w}}_{\mathbf{k}}(\tau_i).q\}=\emptyset,
    \label{e:4.2}
  \end{equation}
  for all $\tau_i$ such that $1\le i<l$.  If $O_{\mathbf{k}}^{\ell}$
  is of length $|\mathbf{k}|$, it is denoted by
  $PO_{\mathbf{k}}^{\ell}$.  A parallel DPS is a pivot strategy that
  for each $\mathbf{k}$ finds $O_{\mathbf{k}}^{\ell}$, given an
  admissible $\ell$.
\end{definition}

Given an admissible $\ell$, $O_{\mathbf{k}}^{\ell}$ from
Definition~\ref{d:4.4} exists and is unique.  Let its length be
$m_{\ell}$.  Then,
$\mathbf{o}(O_{\mathbf{k}}^{\ell})=(\tau_l^{})_{l=1}^{m_{\ell}^{}}$.
Taking the maximal $m_{\ell}$ possible reflects an important choice of
having most possible pivots per each multi-step transformed in
parallel, even though it might imply that $w(O_{\mathbf{k}}^{\ell})$
is smaller than it would have been if only the first
$m_{\ell}'<m_{\ell}^{}$ augmented weights were left in
$O_{\mathbf{k}}^{\ell}$.  Also, $|\mathbf{k}|$ should be considered to
stand for the \emph{desired} number of steps in $\mathbf{k}$, until a
parallel ordering has been found as described below and the actual,
maybe lower, number of steps has been determined.

Algorithm~\ref{a:4.1} sequentially constructs a parallel ordering from
Definition~\ref{d:4.4} for a given index $\ell$ of
$\widehat{\mathbf{w}}_{\mathbf{k}}$.  It can also be constructed in
parallel with Algorithm~\ref{a:4.2}.  Definition~\ref{d:4.4} indicates
validity of Algorithms~\ref{a:4.1} and \ref{a:4.2}.  All parallel
constructs from here on in the paper are the OpenMP~\cite{OpenMP-18}
ones, acting on the shared memory.

\begin{algorithm}[hbtp]
  \SetKw{Break}{break}
  \SetKw{Downto}{downto}
  \KwIn{$\widehat{\mathbf{w}}_{\mathbf{k}}$, $|\mathbf{k}|$, $n_0$, and $\ell$, $1\le\ell\le|\widehat{\mathbf{w}}_{\mathbf{k}}|$.}
  \KwOut{$O_{\mathbf{k}}^{\ell}$ of length at most $|\mathbf{k}|$, $1\le|\mathbf{k}|\le\lfloor n_0/2\rfloor$.}
  \KwData{a Boolean array $\text{\sc r}$ allocated on the thread's stack (faster), with $n_0$ elements initialized to $\bot$.}
  $j:=\ell$\tcp*[r]{start the search within $\widehat{\mathbf{w}}_{\mathbf{k}}$ from the index $\ell$}
  \For(\tcp*[f]{each iteration appends an augmented weight to $O_{\mathbf{k}}^{\ell}$}){$i:=1$ \KwTo $|\mathbf{k}|$}{
    $O_{\mathbf{k}}^{\ell}(i)=\widehat{\mathbf{w}}_{\mathbf{k}}^{}(j)$,\quad$j:=j+1$\tcp*[r]{append the $j$th augmented weight from $\widehat{\mathbf{w}}_{\mathbf{k}}^{}$}
    $\text{\sc r}(O_{\mathbf{k}}^{\ell}(i).p):=\top$,\quad$\text{\sc r}(O_{\mathbf{k}}^{\ell}(i).q):=\top$\tcp*[r]{record $O_{\mathbf{k}}^{\ell}(i).p$ and $O_{\mathbf{k}}^{\ell}(i).q$ as selected}
    \While(\tcp*[f]{search the remaining part of $\widehat{\mathbf{w}}_{\mathbf{k}}$ for the next $\widehat{\mathbf{w}}_{\mathbf{k}}(j)$}){$j\le|\widehat{\mathbf{w}}_{\mathbf{k}}|$}{
      $\text{\sc f}:=\text{\sc r}(\widehat{\mathbf{w}}_{\mathbf{k}}(j).p)\vee\text{\sc r}(\widehat{\mathbf{w}}_{\mathbf{k}}(j).q)$\tcp*[r]{check if $\widehat{\mathbf{w}}_{\mathbf{k}}^{}(j).p$ or $\widehat{\mathbf{w}}_{\mathbf{k}}^{}(j).q$ are selected}
      \lIf(\tcp*[f]{if not, success; exit the loop}){$\neg\text{\sc f}$}{\Break}
      $j:=j+1$\tcp*[r]{else, try with the next $j$}
    }
    \lIf(\tcp*[f]{exit if the next $\widehat{\mathbf{w}}_{\mathbf{k}}(j)$ cannot be found ($|\mathbf{k}|:=i$)}){$j>|\widehat{\mathbf{w}}_{\mathbf{k}}|$}{\Break}
  }
  \caption{Computing a parallel ordering sequentially.}
  \label{a:4.1}
\end{algorithm}

Algorithm~\ref{a:4.2} is the one chosen for the prototype
implementation when $t>1$ (as was the case in the tests), with a
fallback to Algorithm~\ref{a:4.1} when $t=1$.

\begin{algorithm}[hbtp]
  \SetKwInput{Use}{Use}
  \SetKwFor{ParallelFor}{for}{do {\rm in} parallel {\rm with} threads$(t)${\rm ,} shared$(\widehat{\mathbf{w}}_{\mathbf{k}})${\rm ,} reduction$(\min\colon k)$}{end parallel for}
  \SetKw{Break}{break}
  \SetKwFunction{IsNotNaN}{isNotNaN}
  \SetKwFunction{QNaN}{qNaN}
  \Use{\QNaN{$p$} returns a quiet $\mathtt{NaN}$ with its payload set to $p$; $\IsNotNaN{x}=\top$ iff $x$ is not a $\mathtt{NaN}$.}
  \KwIn{$\widehat{\mathbf{w}}_{\mathbf{k}}$, $|\mathbf{k}|$, $n_0$, $t$, and $\ell$, $1\le\ell\le|\widehat{\mathbf{w}}_{\mathbf{k}}|$.}
  \KwOut{$O_{\mathbf{k}}^{\ell}$ of length at most $|\mathbf{k}|$, $1\le|\mathbf{k}|\le\lfloor n_0/2\rfloor$.}
  $i:=\ell$,\quad$l:=0$\tcp*[r]{$i$ is the current index into $\widehat{\mathbf{w}}_{\mathbf{k}}$}
  \While(\tcp*[f]{$l$ is the current length of $O_{\mathbf{k}}^{\ell}$}){$l<|\mathbf{k}|$}{
    $l:=l+1$,\quad$O_{\mathrm{k}}^{\ell}(l)=\widehat{\mathbf{w}}_{\mathbf{k}}^{}(i)$\tcp*[r]{append the current augmented weight to $O_{\mathbf{k}}^{\ell}$}
    \lIf(\tcp*[f]{return $O_{\mathbf{k}}^{\ell}$ if it is complete}){$l\ge|\mathbf{k}|$}{\Break}
    $a_p:=O_{\mathbf{k}}^{\ell}(l).p$,\quad$a_q:=O_{\mathbf{k}}^{\ell}(l).q$\tcp*[r]{shorthands for the indices}
    $k:=|\widehat{\mathbf{w}}_{\mathbf{k}}|+1$\tcp*[r]{$k>i$ will be $\min j$ s.t. $\widehat{\mathbf{w}}_{\mathbf{k}}(j)$ does not collide with $O_{\mathbf{k}}^{\ell}(l)$}
    \ParallelFor{$j=i+1$ \KwTo $|\widehat{\mathbf{w}}_{\mathbf{k}}|$}{
      \If(\tcp*[f]{$\widehat{\mathbf{w}}_{\mathbf{k}}(j)$ is alive if its weight is not a NaN}){\IsNotNaN{$\widehat{\mathbf{w}}_{\mathbf{k}}(j).w$}}{
        $b_p:=\widehat{\mathbf{w}}_{\mathbf{k}}(j).p$,\quad$b_q:=\widehat{\mathbf{w}}_{\mathbf{k}}(j).q$\tcp*[r]{check if\ldots}
        $\text{\sc c}:=(a_p=b_p)\vee(a_p=b_q)\vee(a_q=b_p)\vee(a_q=b_q)$\tcp*[r]{$\widehat{\mathbf{w}}_{\mathbf{k}}(j)$ \&\ $O_{\mathbf{k}}^{\ell}(l)$ collide}
        \leIf(\tcp*[f]{kill $\widehat{\mathbf{w}}_{\mathbf{k}}(j)$ if $\text{\sc c}$}){$\text{\sc c}$}{$\widehat{\mathbf{w}}_{\mathbf{k}}(j).w:=\QNaN{j}$}{$k:=\min\{k,j\}$}
      }
    }
    \leIf(\tcp*[f]{take the lowest surviving index or stop}){$k>|\widehat{\mathbf{w}}_{\mathbf{k}}|$}{\Break}{$i:=k$}
  }
  \caption{Computing a parallel ordering with $t>1$ tasks.}
  \label{a:4.2}
\end{algorithm}

\begin{example}\label{x:4.1}
  \looseness=-1
  Let $A$ be a matrix representation of the
  $\widehat{\mathbf{w}}_{\mathbf{k}}$ computed for some $G_{k-1}$,
  \begin{displaymath}
    \begin{matrix}
      A=\begin{bmatrix}
        \ast_1^{} & 15 & 10 & \hphantom{1}6 & \hphantom{1}3 & {\setlength{\fboxrule}{1.5pt}\fbox{\hphantom{1}1}}\\
        15 & \ast_2^{} & 14 & \hphantom{1}9 & {\setlength{\fboxrule}{1pt}\fbox{\hphantom{1}5}} & \hphantom{1}2\\
        10 & 14 & \ast_3^{} & \fbox{13} & \hphantom{1}8 & \hphantom{1}4\\
        \hphantom{1}6 & \hphantom{1}9 & \fbox{13} & \ast_4^{} & 12 & \hphantom{1}7\\
        \hphantom{1}3 & {\setlength{\fboxrule}{1pt}\fbox{\hphantom{1}5}} & \hphantom{1}8 & 12 & \ast_5^{} & 11\\
        {\setlength{\fboxrule}{1.5pt}\fbox{\hphantom{1}1}} & \hphantom{1}2 & \hphantom{1}4 & \hphantom{1}7 & 11 & \ast_6^{}
      \end{bmatrix},&\hfill&
      \begin{matrix}
        a_{pq}=l\iff\widehat{\mathbf{w}}_{\mathbf{k}}(l)=(w_{pq},p,q);\\\\
        PO_{\mathbf{k}}^1=((w_{16}^{},1,6),(w_{25}^{},2,5),(w_{34}^{},3,4)),\\
        w(PO_{\mathbf{k}}^1)=w_{16}^{}+w_{25}^{}+w_{34}^{},\\
        \mathbf{o}(PO_{\mathbf{k}}^1)=(1,5,13).
      \end{matrix}
    \end{matrix}
  \end{displaymath}
  Here, $a_{pq}=a_{qp}$ for $1\le p<q\le n_0=6$ is the index of
  $w_{pq}^{[\mathbf{k}]}=(w_{pq}^{},p,q)$ in
  $\widehat{\mathbf{w}}_{\mathbf{k}}$.  Then, $PO_{\mathbf{k}}^1$ is
  obtained by either Algorithm~\ref{a:4.1} or \ref{a:4.2}, with the
  pivot index pairs denoted in $A$ as well with the boxes of
  diminishing thickness, corresponding to the decreasing weights.
  Moreover, if all weights are the same (zeros, e.g.), then, by
  Definition~\ref{d:4.2}, $A$ always has the same form as above,
  enlarged or shrunk according to $n_0$, regardless of the actual
  elements in $G_{k-1}$.  Near the end of the Kogbetliantz process,
  where most weights are the same, the pivot strategy predictably
  selects many of the pivots in a pattern resembling the antidiagonal,
  and expanding towards the SE and the NW corners.
\end{example}

\addvspace{-5pt}%
The initial ordering of the elements of $\mathbf{w}_{\mathbf{k}}$ is
irrelevant mathematically.  However, to simplify parallelization of
the iteration over a triangular index space, i.e., over the indices of
the strictly upper triangle of $G_{k-1}$, a one-dimensional array
$\mathcal{I}$ of $\mathbf{n}=n_0(n_0-1)/2$ index pairs is pregenerated
as follows.  Let $i_+:=1$ and $i_-:=\mathbf{n}$.  For each index pair
$(p,q)$ in the column-cyclic order (column-by-column, and
top-to-diagonal within each column), if $J_0(p,p)=J_0(q,q)$ let
$\mathcal{I}(i_+):=(p,q)$ and $i_+:=i_++1$; else, let
$\mathcal{I}(i_-):=(p,q)$ and $i_-:=i_--1$.  This way $\mathcal{I}$ is
partitioned into two contiguous subarrays, $\mathcal{I}_+$ and
$\mathcal{I}_-$ (which might be empty), of the index pairs leading to
the trigonometric or to the hyperbolic transformations, respectively.
One parallel loop over $\mathcal{I}_+$ then computes all the weights
in the trigonometric case, and another parallel loop over
$\mathcal{I}_-$ does the same in the hyperbolic case.  Within each of
the loops all iterations are of equal complexity, what would not have
been the case if the two loops were fused into one, since computing
the weights is essentially faster in the trigonometric ($O(1)$ per
weight) than in the hyperbolic ($O(n_0)$ per weight) case.
\section{Overview of the algorithm}\label{s:5}
In this section the convergence criterion, a vital part of the
$J$-Kogbetliantz algorithm, is discussed in subsection~\ref{ss:5.1},
and the algorithm is summarized in subsection~\ref{ss:5.2}.
\subsection{Convergence criterion}\label{ss:5.1}
Traditionally, a convergence criterion for the Jacobi-like (including
the Kogbetliantz-like) processes is simple and decoupled from the
choice of a pivot strategy.  However, in the hyperbolic case, where
even a small off-diagonal norm can oscillate from one step to another,
a ``global'' stopping criterion, based on the fall of the off-diagonal
norm, relative to the initial one, below a certain threshold, or on
stabilization of the diagonal values, e.g., does not suffice alone for
stopping the process unattended.

The former approach may stop the process when the off-diagonal norm
has relatively diminished, but when there may still be some valid
transformations left that can both change the approximate singular
values and raise the off-diagonal norm.

On the other hand, if the threshold has been set too low, the process
may never (literally or practically) stop and may keep accumulating
the superfluous transformations computed almost exclusively from the
leftover rounding errors off the diagonal.

If a stopping criterion is solely based on observing that the diagonal
elements, i.e., the approximate singular values, have not changed at
all in a sequence of successive steps of a certain, predefined length,
a few conducted tests indicate that the computed singular values are
accurate in a sense of~\eqref{e:6.2} and the diagonal has converged to
its final value, but the singular vectors are not, with the relative
error~\eqref{e:6.1} of order $\sqrt{\varepsilon}$, where $\varepsilon$
is the machine precision, i.e., the transformations left unperformed
would have contributed to the singular vectors significantly, but not
to the singular values.

A ``local'' convergence criterion is thus needed, based on the
$2\times 2$ transformations in a multi-step belonging to a narrow
class of matrices, as shown in Algorithm~\ref{a:5.1}.

\begin{algorithm}[hbtp]
  \SetKw{Stop}{stop}
  $b:=0$\tcp*[r]{a counter of the \emph{big} steps}
  \ForEach(\tcp*[f]{assume the $k$th transformations have been computed}){$k\in\mathbf{k}$}{
    \leIf{{\rm$\widehat{G}_{k-1}$ was diagonal}}{$\text{\sc d}=\top$}{$\text{\sc d}=\bot$}
    \leIf{{\rm$\widehat{U}_k^{\ast}$ is identity}}{$\text{\sc u}=\top$}{$\text{\sc u}=\bot$}
    \leIf{{\rm$\widehat{V}_k$ is identity}}{$\text{\sc v}=\top$}{$\text{\sc v}=\bot$}
    $\text{\sc s}=\text{\sc d}\vee(\text{\sc u}\wedge\text{\sc v})$\tcp*[r]{$\text{\sc s}$ is $\top$ if the $k$th step is \emph{small}}
    \lIf(\tcp*[f]{else, the $k$th step is \emph{big}}){$\neg\text{\sc s}$}{$b:=b+1$}
  }
  \lIf(\tcp*[f]{halt if no \emph{big} steps in $\mathbf{k}$}){$b=0$}{\Stop}
  \caption{The convergence criterion, evaluated at each multi-step $\mathbf{k}$.}
  \label{a:5.1}
\end{algorithm}

The steps of each multi-step $\mathbf{k}$ are categorized as either
\emph{big} or \emph{small}.  A step is \emph{big} if its $2\times 2$
pivot submatrix is not diagonal, and either the left or the right
transformation is not identity; else, it is \emph{small}.  A
non-trivial \emph{small} step is just a scaling by the factors of unit
modulus and/or a swap of the diagonal elements, so it is a heuristic
but reasonable expectation that an absence of \emph{big} steps is an
indication of convergence.
\subsection{The $J$-Kogbetliantz algorithm}\label{ss:5.2}
\looseness=-1
The $J$-Kogbetliantz algorithm is summarized in Algorithm~\ref{a:5.2}.
Note that accumulating the left and the right singular vectors is
optional, and that $\Sigma$ is the diagonal of $G_N$.

\begin{algorithm}[hbtp]
  \SetKw{Break}{break}
  \KwIn{$G_0$ and $J_0$, preprocessed from $G$ and $J$, if necessary.}
  \KwOut{$N$, $U$, $\Sigma$, and $V^{-1}$, as described in section~\ref{s:1}.}
  $N:=0$\tcp*[r]{a counter of steps performed in the loop below}
  \Repeat(\tcp*[f]{the loop body is a multi-step $\mathbf{k}$}){convergence detected by Algorithm~\ref{a:5.1}}{
    compute the augmented weights $\mathbf{w}_{\mathbf{k}}$ as described in subsection~\ref{ss:3.2} and section~\ref{s:4}\;
    filter and sort $\mathbf{w}_{\mathbf{k}}$ to obtain $\widehat{\mathbf{w}}_{\mathbf{k}}$ as described in subsection~\ref{ss:4.2}\;
    \lIf(\tcp*[f]{terminate early, if possible}){$|\widehat{\mathbf{w}}_{\mathbf{k}}|=0$}{\Break}
    generate the parallel ordering $O_{\mathbf{k}}^1$ by Algorithm~\ref{a:4.2} (or by Algorithm~\ref{a:4.1} if $t=1$)\;
    \ForEach(\tcp*[f]{in parallel with $t$ threads}){$k\in\mathbf{k}$}{
      compute (or reuse) the $2\times 2$ HSVD of $\widehat{G}_{k-1}^{}$, i.e., $\widehat{U}_k^{\ast}$, $\widehat{\Sigma}_k^{}$, and $\widehat{V}_k^{}$, as in section~\ref{s:2}\;
      apply $\widehat{U}_k^{\ast}$ from the left to $U_{k-1}^{\ast}$ and $G_{k-1}^{}$ to obtain $U_k^{\ast}$ and $G_{k-1}'$, resp., as in section~\ref{s:3}\;
    }\ForEach(\tcp*[f]{in parallel with $t$ threads}){$k\in\mathbf{k}$}{
      apply $\widehat{V}_k^{}$ from the right to $G_{k-1}'$ and $V_{k-1}^{}$ to obtain $G_k^{}$ and $V_k^{}$, resp., as in section~\ref{s:3}\;
    }
    $N:=N+|\mathbf{k}|$\tcp*[r]{end of the multi-step $\mathbf{k}$}
  }
  compute $U=(U_N^{\ast})^{\ast}$ and $V^{-1}=J_0^{}V_N^{\ast}J_0^{}$\tcp*[r]{optionally}
  \caption{Overview of the $J$-Kogbetliantz algorithm.}
  \label{a:5.2}
\end{algorithm}
\section{Numerical testing}\label{s:6}
Testing was performed on the Intel Xeon Phi 7210 CPUs, running at
$1.3$~GHz with TurboBoost turned off in Quadrant cluster mode, with
$96$~GiB of RAM and $16$~GiB of flat-mode MCDRAM (which was not used),
under $64$-bit CentOS Linux $7.9.2009$ with the Intel compilers
(Fortran, C), version $19.1.3.304$, and the GNU compilers (Fortran,
C), version $9.3.1$, for the error checking.  No BLAS/LAPACK routines
from Intel Math Kernel Library were used in the final prototype
implementation.

The prototype code has been written in Fortran for the
\texttt{DOUBLE PRECISION} and \texttt{DOUBLE COMPLEX} datatypes, with
some auxiliary routines written in C\@.  The real and the complex
$J$-Kogbetliantz algorithms are implemented as two programs.  There
are also two error checkers in quadruple precision (Fortran's
\texttt{KIND=REAL128}), one which finds the absolute and then the
relative normwise error of the obtained HSVD as
\begin{equation}
  \|G_0-U\Sigma V^{-1}\|_F/\|G_0\|_F,
  \label{e:6.1}
\end{equation}
while the other compares $\Sigma J_0 \Sigma^T$ with the eigenvalues
$\Lambda$ of $H=G_0^{} J_0^{} G_0^{\ast}$, i.e.,
\begin{equation}
  \max_{1\le i\le n_0^{}}^{}|(\lambda_{ii}'-\sigma_{ii}^2 j_{ii}^{})/\lambda_{ii}'|,\quad
  \lambda_{11}'\ge\lambda_{22}'\ge\ldots\ge\lambda_{n_0^{}n_0^{}}',
  \label{e:6.2}
\end{equation}
where $\Lambda'=P_{\Lambda}^{}\Lambda P_{\Lambda}^T$ ($P_{\Lambda}$
being a permutation) has the eigenvalues on the diagonal sorted
descendingly to match the ordering of $\Sigma J_0 \Sigma^T$.  All
eigenvalues are non-zero.

The close-to-exact eigenvalues are known since each $H$ has been
generated by taking its double precision eigenvalues $\Lambda$
pseudorandomly from one of the ranges:
\begin{compactenum}
\item $\lambda\in\langle\epsilon,1]$, drawn uniformly from $\langle 0,1]$,
\item $|\lambda|\in\langle\epsilon,1]$, drawn uniformly from $[-1,1]$,
\item $|\lambda|\in\langle\epsilon,1]$, drawn from the normal variable
  $\mathcal{N}(\mu=0,\sigma=1)$,
\end{compactenum}
with a given
$\epsilon\in\{\epsilon_1=10^{-13},\epsilon_2=10^{-15}\}$.  Then,
$H=U\Lambda U^{\ast}$ (or $U\Lambda U^T$) is formed by applying
$n_0-1$ pseudorandom Householder reflectors to $\Lambda$ in extended
precision.

The Hermitian/symmetric indefinite factorization with complete
pivoting~\cite{Singer-DiNapoli-Novakovic-Caklovic-20,Slapnicar-98} of
$H$ gives $J_0$ and $G_0'$, which is rounded to a double
(complex/real) precision input $G_0$.  For each $n_0$ twelve pairs
$(G_0,J_0)$ have been generated, six each for the real and the complex
case.  In each case two pairs come with $J_0=I_{n_0}$, corresponding
to the first range above.  For a given $n_0$, $\epsilon$, and a range,
the eigenvalues of the real $H$ are the same as those of the complex
$H$, due to a fixed pseudo-RNG seed selected for that $\epsilon$.

\looseness=-1
For each field $T\in\{\mathbb{R},\mathbb{C}\}$, range $L\in\{1,2,3\}$
of the eigenvalues of $H$, and $\epsilon$ as above, a sequence of test
matrices was generated, with their orders ranging from $n_0=4$ to
$n_0=2048$ with a variable step: four up to $n_0=128$, eight up to
$n_0=256$, $16$ up to $n_0=512$, $32$ up to $n_0=1024$, and $256$
onwards.  For each $n_0$, the number of tasks for a run of the
$J$-Kogbetliantz algorithm was $t=\min\{64,n_0(n_0-1)/2\}$, since the
CPU has 64 cores, and to each core at most one task (i.e., an OpenMP
thread) was assigned by setting \texttt{OMP\_PLACES=CORES} and
\texttt{OMP\_PROC\_BIND=SPREAD} environment variables.

Let $\mathbf{N}$, $0\le\mathbf{N}\le N$, be the number of multi-steps
performed until convergence.  Then, define $\mathbf{C}$, the number of
`virtual' sweeps (also called cycles) performed, as
\begin{equation}
  \mathbf{C}=\mathbf{N}/(n_0-1).
  \label{e:6.3}
\end{equation}
A `virtual' sweep has at most the same number of steps as would a
`real' sweep by a cyclic pivot strategy have, i.e., $n_0(n_0-1)/2$,
but in it any transformation candidate can be transformed up to
$\lfloor n_0/2\rfloor$ times.  Note that $\mathbf{C}$ does not have to
be an integer.

In each subfigure of Figures~\ref{f:6.1}--\ref{f:6.3} there are three
data series, one for each $L$.  A data point in a series is the
maximum of a value from one run of the $J$-Kogbetliantz algorithm on a
matrix generated with $\epsilon=\epsilon_1$, and a value from another
run on a matrix generated with $\epsilon=\epsilon_2$, with all other
parameters (i.e., $T$, $L$, and $n_0$) being the same.
\begin{figure}[hbt]
  \begin{center}
    \includegraphics{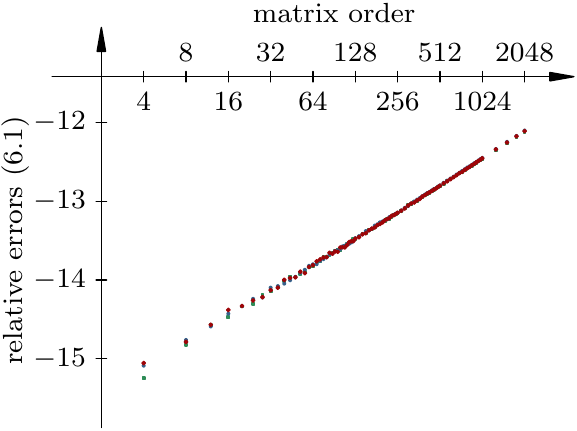}\hfill
    \includegraphics{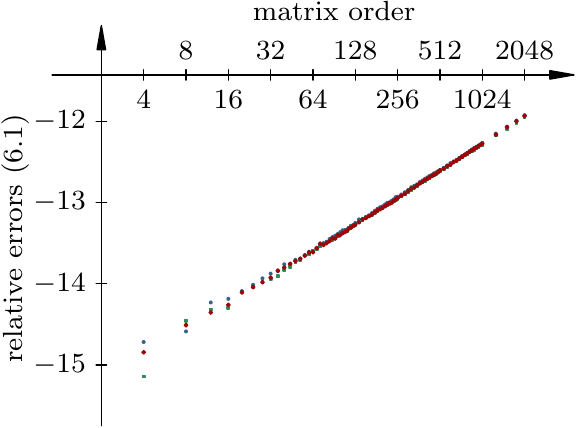}\\
    \includegraphics{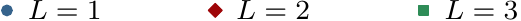}
  \end{center}
  \caption{The relative errors~\eqref{e:6.1}, in $\log_{10}$-scale, in
    the HSVD computed in \texttt{DOUBLE PRECISION} (left) and
    \texttt{DOUBLE COMPLEX} (right) datatypes.  The matrix orders on
    $x$-axis are in $\log_2$-scale.}
  \label{f:6.1}
\end{figure}
\begin{figure}[hbt]
  \begin{center}
    \includegraphics{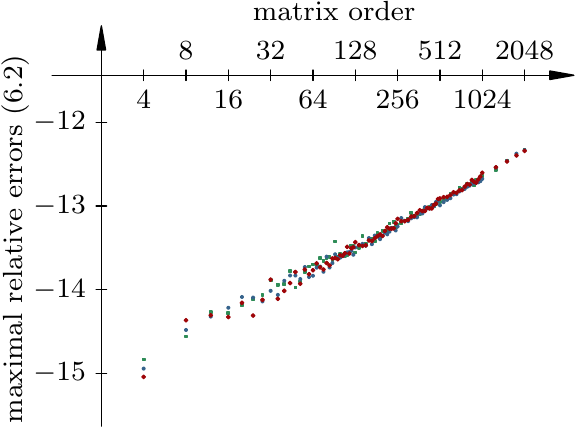}\hfill
    \includegraphics{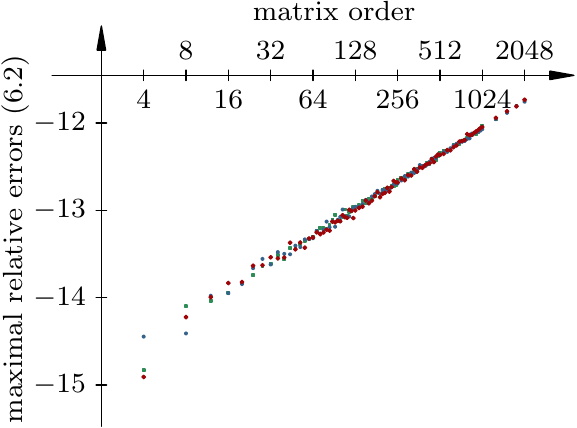}\\
    \includegraphics{fig7.pdf}
  \end{center}
  \caption{The maximal relative errors~\eqref{e:6.2}, in
    $\log_{10}$-scale, in the eigenvalues of $G_0^{}J_0^{}G_0^{\ast}$,
    with the $J_0$-HSVD of $G_0$ computed in \texttt{DOUBLE PRECISION}
    (left) and \texttt{DOUBLE COMPLEX} (right) datatypes.  The matrix
    orders on $x$-axis are in $\log_2$-scale.}
  \label{f:6.2}
\end{figure}
\begin{figure}[hbt]
  \begin{center}
    \includegraphics{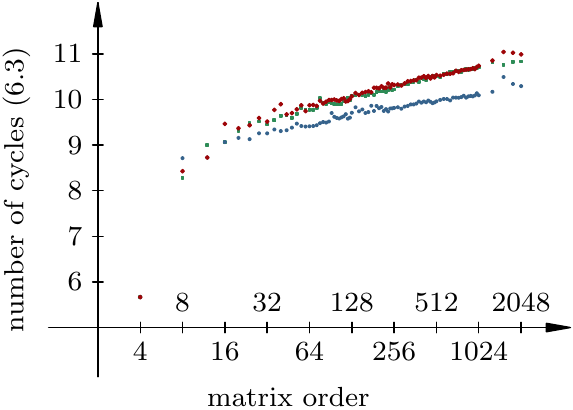}\hfill
    \includegraphics{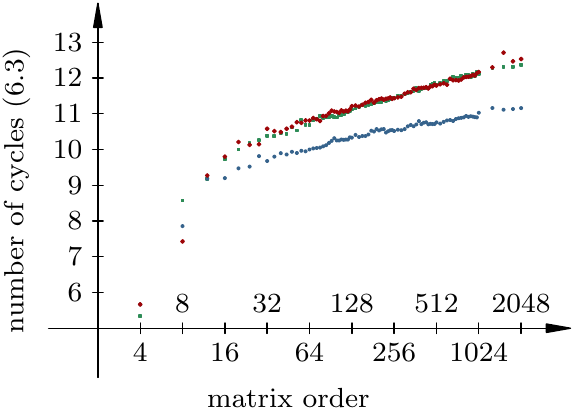}\\
    \includegraphics{fig7.pdf}
  \end{center}
  \caption{The number of cycles~\eqref{e:6.3} until convergence, when
    computing in \texttt{DOUBLE PRECISION} (left) and
    \texttt{DOUBLE COMPLEX} (right) datatypes.  The matrix orders on
    $x$-axis are in $\log_2$-scale.}
  \label{f:6.3}
\end{figure}

A comparison of the relative errors in the decomposition, shown in
Figure~\ref{f:6.1}, leads to a similar conclusion that can be reached
by comparing the maximal relative errors in the eigenvalues of $H$,
shown in Figure~\ref{f:6.2}.  In the real as well as in the complex
case a satisfactory accuracy, in a sense of both~\eqref{e:6.1} and
\eqref{e:6.2}, was reached in a reasonably small number of cycles, as
shown in Figure~\ref{f:6.3}.

A sequential (with $t=1$) variant of the algorithm, performing one
step at a time, was compared performance-wise against the parallel
multi-step one, with $\upsilon=0.75$ and $n_0$ going up to $128$ and
$256$ for $T=\mathbb{C}$ and $T=\mathbb{R}$, respectively.  The
sequential variant was drastically slowed down (up to more than three
orders of magnitude) compared to the multi-step one, especially for
the ``true'' HSVD (less so for the ``ordinary'' SVD), to a point of
being totally impractical.  The $J$-Kogbetliantz algorithm is
therefore best run in the multi-step regime, with as much parallelism
as possible.
\section{Conclusions and future work}\label{s:7}
In this paper an accurate method for computing the $2\times 2$ HSVD
of real and complex matrices is demonstrated and employed as one of
the three major building blocks of a $J$-Kogbetliantz algorithm for
general square matrices.  The other two important contributions are a
heuristic but efficient convergence criterion for all pointwise
Kogbetliantz-type processes and a modification of the well-established
dynamic pivot strategy that can cope with the pivot weights of
arbitrary magnitudes and signs.

To keep the exposition concise, a forward rounding error analysis of
the floating-point computation of the $2\times 2$ HSVD is left for
future work.  Furthermore, performing such analysis is slightly
impeded by, e.g., a lack of standardized, tight error bounds for the
absolute value of a complex number, or equivalently, of the
$\mathtt{HYPOT}$ intrinsic.

The $J$-Kogbetliantz algorithm, as presented, is not highly performant
even in its parallel form.  It is worth exploring if (and what kind
of) blocking of the algorithm would be beneficial in terms of
performance, without negatively affecting accuracy.  A straightforward
generalization of the dynamic pivot strategy to block (instead of
$2\times 2$) pivots seems too inefficient, as it would assume at each
block-multi-step the full diagonalization of all possible block pivots
that require hyperbolic transformations only to compute their weights.
Apart from---and complementary to---blocking, there are other options
for improving performance, like storing and reusing the $2\times 2$
HSVDs computed while forming a multi-step, as explained in
section~\ref{s:3}, and implementing a vectorized sorting routine for
the suitably represented augmented weights.

The batches of $2\times 2$ transformations in each multi-step could be
processed in a vectorized way, as in~\cite{Novakovic-20}, should a
highly optimized implementation be required.  Such a version of the
algorithm would have to resort to a specific vectorized routine in
each of the three cases of $2\times 2$ transformations (one
trigonometric and two hyperbolic, of lower and upper triangular
matrices).  Consequently, (up to) three disjoint batches would have to
be processed separately, with a non-trivial repacking of input data
for each of them.  A serious practical use-case is required to justify
such effort.

Finally, an interesting observation is offered without a proof.  For
complex matrices, the $J$-Kogbetliantz algorithm seems to converge (in
limit) not only in the sense of $\off(G_k)\to 0$ and
$\diag(G_k)\to\Sigma$, but also with
$\displaystyle\max_{i\ne j}|\arg(G_k)_{ij}|\to 0$, for $k\to\infty$.
\begin{acknowledgements}
  We are much indebted to Sa\v{s}a Singer$\null^{\dagger}$ for his
  suggestions on the paper's subject, and to the anonymous referee for
  significantly improving the presentation of the paper.
\end{acknowledgements}
\section*{Declarations}
\def\ackname{Funding}
\begin{acknowledgements}
  This work has been supported in part by Croatian Science Foundation
  under the project IP--2014--09--3670
  ``Matrix Factorizations and Block Diagonalization Algorithms''
  (\href{https://web.math.pmf.unizg.hr/mfbda/}{MFBDA}).
\end{acknowledgements}
\def\ackname{Conflicts of Interest}
\begin{acknowledgements}\addvspace{-14pt}%
  The authors have no conflicts of interest to declare that are
  relevant to the content of this article.
\end{acknowledgements}
\def\ackname{Availability of data}
\begin{acknowledgements}\addvspace{-14pt}%
  The input dataset is available at
  \url{https://pmfhr-my.sharepoint.com/:f:/g/personal/venovako_math_pmf_hr/EiBdopxcJfFHvCibaJrSpDoBL2hSMfGLREP0npJKu_bIKg}.
\end{acknowledgements}
\def\ackname{Code availability}
\begin{acknowledgements}\addvspace{-14pt}%
  The source code is available in
  \url{https://github.com/venovako/JKogb} repository.
\end{acknowledgements}
\def\ackname{Authors' contributions}
\begin{acknowledgement}\addvspace{-14pt}%
  The second author formulated the research topic, reviewed the
  literature, plotted the figures and proofread the manuscript.  The
  first author performed the rest of the research tasks.
\end{acknowledgement}
%

%
\appendix
\section{Proofs of Lemmas~\ref{l:3.1} and~\ref{l:3.2}}\label{s:A}
\begin{proof}[Lemma~\ref{l:3.1}]
  Let, for $1\le\ell\le n$, $\gamma_{\ell}=\arg(x_{\ell})$ and
  $\delta_{\ell}=\arg(y_{\ell})$.  Then,
  \begin{displaymath}
    \renewcommand{\arraycolsep}{3pt}
    \begin{bmatrix}x_{\ell}&y_{\ell}\end{bmatrix}=
    \begin{bmatrix}e^{\mathrm{i}\gamma_{\ell}}|x_{\ell}|&e^{\mathrm{i}\delta_{\ell}}|y_{\ell}|\end{bmatrix}=
    e^{\mathrm{i}\gamma_{\ell}}\begin{bmatrix}|x_{\ell}|&e^{\mathrm{i}(\delta_{\ell}-\gamma_{\ell})}|y_{\ell}|\end{bmatrix}=
    e^{\mathrm{i}\delta_{\ell}}\begin{bmatrix}e^{\mathrm{i}(\gamma_{\ell}-\delta_{\ell})}|x_{\ell}|&|y_{\ell}|\end{bmatrix}.
  \end{displaymath}
  Using the second equality, from the matrix multiplication it follows
  \begin{displaymath}
    x_{\ell}'=e^{\mathrm{i}\gamma_{\ell}^{}}(|x_{\ell}^{}|\cosh\psi+e^{\mathrm{i}(\delta_{\ell}^{}-\gamma_{\ell}^{}-\beta)}|y_{\ell}^{}|\sinh\psi),
  \end{displaymath}
  and using the third equality, from the matrix multiplication it follows
  \begin{displaymath}
    y_{\ell}'=e^{\mathrm{i}\delta_{\ell}^{}}(e^{\mathrm{i}(\gamma_{\ell}^{}-\delta_{\ell}^{}+\beta)}|x_{\ell}^{}|\sinh\psi+|y_{\ell}^{}|\cosh\psi).
  \end{displaymath}
  Since $|e^{-\mathrm{i}\gamma_{\ell}^{}}x_{\ell}'|=|x_{\ell}'|$ and
  $|e^{-\mathrm{i}\delta_{\ell}^{}}y_{\ell}'|=|y_{\ell}'|$ and
  $\cos(-\phi)=\cos\phi$, it holds
  \begin{displaymath}
    \begin{aligned}
      |x_{\ell}'|^2
      &=(|x_{\ell}^{}|\cosh\psi+\cos(\delta_{\ell}^{}-\gamma_{\ell}^{}-\beta)|y_{\ell}^{}|\sinh\psi)^2+(\sin(\delta_{\ell}^{}-\gamma_{\ell}^{}-\beta)|y_{\ell}^{}|\sinh\psi)^2\\
      &=|x_{\ell}^{}|^2\cosh^2\psi+|y_{\ell}^{}|^2\sinh^2\psi+\cos(\delta_{\ell}^{}-\gamma_{\ell}^{}-\beta)|x_{\ell}^{}||y_{\ell}^{}|2\cosh\psi\sinh\psi,\\
      |y_{\ell}'|^2
      &=(\cos(\gamma_{\ell}^{}-\delta_{\ell}^{}+\beta)|x_{\ell}^{}|\sinh\psi+|y_{\ell}^{}|\cosh\psi)^2+(\sin(\gamma_{\ell}^{}-\delta_{\ell}^{}+\beta)|x_{\ell}^{}|\sinh\psi)^2\\
      &=|x_{\ell}^{}|^2\sinh^2\psi+|y_{\ell}^{}|^2\cosh^2\psi+\cos(\delta_{\ell}^{}-\gamma_{\ell}^{}-\beta)|x_{\ell}^{}||y_{\ell}^{}|2\cosh\psi\sinh\psi.
    \end{aligned}
  \end{displaymath}
  After grouping the terms, the square of the Frobenius norm of the new
  $\ell$th row is
  \begin{equation}
    \begin{aligned}
      \left\|\begin{bmatrix}x_{\ell}' & y_{\ell}'\end{bmatrix}\right\|_F^2=
      |x_{\ell}'|^2+|y_{\ell}'|^2&=
      (\cosh^2\psi+\sinh^2\psi)(|x_{\ell}^{}|^2+|y_{\ell}^{}|^2)\\
      &+2\cos(\delta_{\ell}^{}-\gamma_{\ell}^{}-\beta)|x_{\ell}^{}||y_{\ell}^{}|2\cosh\psi\sinh\psi.
    \end{aligned}
    \label{e:A.1}
  \end{equation}
  Summing the left side of the equation~\eqref{e:A.1} over all $\ell$
  one obtains
  \begin{displaymath}
    \left\|\begin{bmatrix}\mathbf{x}'&\mathbf{y}'\end{bmatrix}\right\|_F^2=
    \sum_{\ell=1}^n\left(|x_{\ell}'|^2+|y_{\ell}'|^2\right),
  \end{displaymath}
  what is equal to the right side of the equation~\eqref{e:A.1},
  summed over all $\ell$,
  \begin{displaymath}
    \sum_{\ell=1}^n\left((\cosh^2\psi+\sinh^2\psi)(|x_{\ell}^{}|^2+|y_{\ell}^{}|^2)+2\zeta_{\ell}^{}|x_{\ell}^{}||y_{\ell}^{}|2\cosh\psi\sinh\psi\right),
  \end{displaymath}
  where
  $-1\le\zeta_{\ell}^{}=\cos(\delta_{\ell}^{}-\gamma_{\ell}^{}-\beta)\le 1$,
  so $|\zeta_{\ell}^{}|\le 1$.  The last sum can be split into a
  non-negative part and the remaining part of an arbitrary sign,
  \begin{displaymath}
    (\cosh^2\psi+\sinh^2\psi)\sum_{\ell=1}^n(|x_{\ell}^{}|^2+|y_{\ell}^{}|^2)+
    2\cosh\psi\sinh\psi\sum_{\ell=1}^n 2\zeta_{\ell}^{}|x_{\ell}^{}||y_{\ell}^{}|.
  \end{displaymath}
  Using the triangle inequality, and observing that
  $\sum_{\ell=1}^n(|x_{\ell}^{}|^2+|y_{\ell}^{}|^2)=\left\|\begin{bmatrix}\mathbf{x}&\mathbf{y}\end{bmatrix}\right\|_F^2$,
  this value can be bounded above by
  \begin{displaymath}
    (\cosh^2\psi+\sinh^2\psi)\left\|\begin{bmatrix}\mathbf{x}&\mathbf{y}\end{bmatrix}\right\|_F^2+
    2\cosh\psi|\sinh\psi|\sum_{\ell=1}^n 2|x_{\ell}^{}||y_{\ell}^{}|,
  \end{displaymath}
  and below by
  \begin{displaymath}
    (\cosh^2\psi+\sinh^2\psi)\left\|\begin{bmatrix}\mathbf{x}&\mathbf{y}\end{bmatrix}\right\|_F^2-
    2\cosh\psi|\sinh\psi|\sum_{\ell=1}^n 2|x_{\ell}^{}||y_{\ell}^{}|,
  \end{displaymath}
  where both bounds can be simplified by the identities
  \begin{displaymath}
    \cosh^2\psi+\sinh^2\psi=\cosh(2\psi),\qquad
    2\cosh\psi|\sinh\psi|=|\sinh(2\psi)|.
  \end{displaymath}
  By the inequality of arithmetic and geometric means it holds
  $2|x_{\ell}^{}||y_{\ell}^{}|\le(|x_{\ell}^{}|^2+|y_{\ell}^{}|^2)$,
  so a further upper bound is reached as
  \begin{displaymath}
    \cosh(2\psi)\left\|\begin{bmatrix}\mathbf{x}&\mathbf{y}\end{bmatrix}\right\|_F^2+
    |\sinh(2\psi)|\left\|\begin{bmatrix}\mathbf{x}&\mathbf{y}\end{bmatrix}\right\|_F^2,
  \end{displaymath}
  and a further lower bound as
  \begin{displaymath}
    \cosh(2\psi)\left\|\begin{bmatrix}\mathbf{x}&\mathbf{y}\end{bmatrix}\right\|_F^2-
    |\sinh(2\psi)|\left\|\begin{bmatrix}\mathbf{x}&\mathbf{y}\end{bmatrix}\right\|_F^2,
  \end{displaymath}
  what, after grouping the terms and dividing by
  $\left\|\begin{bmatrix}\mathbf{x}&\mathbf{y}\end{bmatrix}\right\|_F^2$,
  concludes the proof.
  \qed%
\end{proof}
\begin{proof}[Lemma~\ref{l:3.2}]
  Note that $\cosh(2\psi)+|\sinh(2\psi)|\ge\cosh(2\psi)\ge 1$, and
  \begin{displaymath}
    \begin{aligned}
      1&=\cosh^2(2\psi)-\sinh^2(2\psi)\\
      &=(\cosh(2\psi)-|\sinh(2\psi)|)\cdot(\cosh(2\psi)+|\sinh(2\psi)|)\\
      &\ge\cosh(2\psi)-|\sinh(2\psi)|>0.
    \end{aligned}
  \end{displaymath}
  If $\psi=0$, the equalities in the bounds established in
  Lemma~\ref{l:3.1} hold trivially.  Also, if both equalities hold
  simultaneously, $\psi=0$.

  The inequality of arithmetic and geometric means in the proof of
  Lemma~\ref{l:3.1} turns into equality if and only if
  $|x_{\ell}|=|y_{\ell}|$ for all $\ell$.  When
  $|x_{\ell}||y_{\ell}|\ne 0$, it has to hold
  $\zeta_{\ell}=\zeta$, where $\zeta=\pm\sign(\sinh\psi)$, to reach
  the upper or the lower bound, respectively.  From $\zeta=\pm 1$ it
  follows $\delta_{\ell}=\gamma_{\ell}+\beta+l\pi$ for a fixed
  $l\in\mathbb{Z}$, i.e.,
  $x_{\ell}=e^{\mathrm{i}\gamma_{\ell}}|x_{\ell}|$ and
  $y_{\ell}=\pm e^{\mathrm{i}\beta}e^{\mathrm{i}\gamma_{\ell}}|x_{\ell}|$,
  so $y_{\ell}=\pm e^{\mathrm{i}\beta}x_{\ell}$ for all $\ell$.
  Conversely, $\mathbf{y}=\pm e^{\mathrm{i}\beta}\mathbf{x}$ implies,
  for all $\ell$, that $|x_{\ell}|=|y_{\ell}|$ and $\zeta_{\ell}$ is a
  constant $\zeta=\pm 1$, so one of the two bounds is reached.
  \qed%
\end{proof}
\def\ackname{Sanja Singer\hskip-2pt}
\begin{acknowledgement}
  (1963--2021) was a tenured professor of Mathematics at the Faculty
  of Mechanical Engineering and Naval Architecture, University of
  Zagreb, Croatia.  She received a Ph.D.\ in Mathematics in 1997 from
  the same University.  Her research was focused on accurate and
  high-performance algorithms of numerical linear algebra, especially
  matrix factorization algorithms, and eigenvalue and singular value
  algorithms for dense matrices.  She educated many generations of
  engineers and mathematicians, introduced several undergraduate and
  graduate courses in numerical mathematics and parallel computing,
  and supervised two grateful doctoral students, while selflessly
  helping countless others.
\end{acknowledgement}
\end{document}